\documentclass[12pt]{amsart}
\usepackage{note}
\usepackage[utf8]{inputenc}
\usepackage{here}
\usepackage{ytableau}
\usepackage{graphicx}
\usepackage{multicol}

\usepackage[margin=1in]{geometry}

\title{On $A$-parameters containing unitary lowest weight representations of $\mathrm{U}(p,q)$}
\author{Shuji Horinaga}
\email{syuuji.horinaga@ntt.com, shorinaga@gmail.com}

\begin{document}

\newlength\scratchlength
\newcommand\s[2]{
  \settoheight\scratchlength{\mathstrut}%
  \scratchlength=\number\numexpr\number#1-1\relax\scratchlength
  \lower.5\scratchlength\hbox{\scalebox{1}[#1]{$#2$}}%
}
\renewcommand{\labelenumi}{(\arabic{enumi})}
\begin{abstract}
    In this paper, we determine all the Arthur packets containing an irreducible unitary lowest weight representation $\pi$ of real unitary group $G = \mathop{\mathrm{U}}(p,q)$, including non-scalar cases.
    Our methods are the Barbasch-Vogan parametrization of representations of $G$ and Trapa's algorithm to calculate the cohomologically induced representations.
    In particular, we show that an Arthur packet has at most one irreducible unitary lowest weight representation of $G$.
    As a consequence, if an irreducible unitary lowest weight representation $\pi$ exists in the Arthur packet of $\psi$, we give an explicit formula of the lowest $K$-type of $\pi$.
\end{abstract}
\maketitle
\section{Introduction}
    A central problem in the theory of automorphic forms is the endoscopic classification of irreducible cuspidal automorphic representations and counting the multiplicities.
    Arthur's multiplicity formula, which gives the multiplicity formula in terms of irreducible cuspidal automorphic representations of general linear groups, has been proven in several cases \cite{Arthur_2013_book, Mok_2015, KMSW_endoscopic_classification_unitary, Chen_Zou_2024_AMF} and references therein.
    These prior works have relied on Arthur's unpublished works, but they have come to be proved (cf.~\cite{AGIKMS_LIR_2024, LLS_2024_Aubert_dual_lemma_Arthur}).
    
    In Arthur's multiplicity formula, the irreducible discrete automorphic representations are given by the restricted tensor products of representations in local Arthur packets (or $A$-packets) with characters of component groups satisfying the product formula.
    In the classical language, Arthur's multiplicity formula reveals the counting of dimensions or the existence of square-integrable holomorphic modular forms.
    To calculate these, we should consider the Arthur packet at archimedean places containing lowest weight representations corresponding to the weights of modular forms under consideration.
    M\oe glin-Renard \cite{MR_consequence} provides the explicit description of Arthur packets in terms of cohomological inductions $A_\frakq(\lambda)$ that will be written by $\scrA_{\underline{d}}(\psi)$, but it remains to consider what the representation $\scrA_{\underline{d}}(\psi)$ is.
    One option is to calculate the nonvanishing condition of $\scrA_{\underline{d}}(\psi)$ by Trapa's algorithm \cite{Trapa_2001}.
    More recently, the nonvanishing condition has been given by a closed form \cite[Theorem 2.8]{C-Huang_2024}.
    In the current work, we describe all the Arthur packets containing a given lowest weight representation of real unitary groups.
    This essentially reveals all the $\scrA_{\underline{d}}(\psi)$ in an Arthur packet, which is isomorphic to the given lowest weight representation.
    In the case of $\Sp_{2n}(\bbR)$, the similar but only for scalar lowest weight case is discussed in \cite{MR_ht_wt}.

    To state the main theorem, we recall lowest weight representations and M\oe glin-Renard's description of Arthur packets.
    Let $G = \U(p,q)$ with $N= p+q$ and $\frakg_\bbC$ be the complexified Lie algebra of $G$.
    Choose a Cartan subgroup $T$ as the group of diagonal matrices of $G$.
    We identify $\frakt_\bbC = \Lie(T) \otimes_\bbR \bbC$ as $\bbC^N$ by a basis $\{e_1,\ldots,e_N\}$ defined by $e_i (\diag(t_1,\ldots,t_N)) = t_i$.
    Then, the root system $\triangle$ is given by
    \[
    \triangle = \{\pm(e_i - e_j) \mid 1 \leq i < j \leq N\}.
    \]
    We choose the positive system $\triangle^+$ as
    \[
    \triangle^+ = \{e_i - e_j \mid 1\leq i < j \leq N\}.
    \]
    Let $\frakb$ be the Borel subalgebra of $\frakg_\bbC$ corresponding to $\triangle^+$ and $\frakb^-$ be the opposite of $\frakb$.
    With the same choice of positive roots, the highest weights of the maximal compact subgroup $K$ under the usual choice are $(\lambda_1,\ldots,\lambda_{p}, \lambda_{p+1}, \ldots,\lambda_N) \in \bbZ^N$ satisfying $\lambda_1 \geq \cdots \lambda_p, \lambda_{p+1}\geq \cdots \geq \lambda_{N}$.
    We denote by $\lambda$ the irreducible representation of $K$ with highest weight $\lambda$ for short.
    We say that a $(\frakg, K)$-module is a lowest weight representation if there exists $v \in \pi$ such that the vector $v$ generates $\pi$ and $v$ is annihilated by $\frakb^-$.
    The vector $v$ is called a lowest weight vector.
    For an irreducible lowest weight representation $\pi$, there exists a unique representation $\tau$ of $K$ called the lowest $K$-type such that a lowest weight vector generates $\tau$ as a representation of $K$.
    For each irreducible representation of $K$ with highest weight $\lambda = (\lambda_1,\ldots,\lambda_p, \lambda_{p+1},\ldots,\lambda_N) \in \bbZ^N$, there exists a unique irreducible lowest weight representation $\pi_\lambda$ with the lowest $K$-type $\lambda$.
    For an irreducible representation $\pi$, let $\chi_\pi$ be the infinitesimal character of $\pi$.
    
    For $A$-parameters, we only consider good $A$-parameters for short.
    The good $A$-parameters $\psi$ of $G$ can be viewed as a formal sum
    \[
    \psi =  \bigoplus_{i=1}^r \chi_{t_i} \otimes S_{a_i}
    \]
    such that $\sum_{i} a_i = N$, $S_m$ is the irreducible representation of $\SL_2(\bbC)$ with dimension $m$ and $t_i + a_i + N \in 2\bbZ$, where $\chi_{t}$ is the character of $\bbC^\times$ defined by $\chi_{t} (z) = z^{t/2}\overline{z}^{-t/2}$.
    Suppose $t_i \geq t_{i+1}$, and, $a_{i} \geq a_{i+1}$ if $t_i = t_{i+1}$.
    Put
    \[
    \calD(\psi) = \left\{(p_i, q_i) \in (\bbZ_{\geq 0} \times \bbZ_{\geq 0})^r \,\middle|\, p_i + q_i = a_i, \sum_{i=1}^r p_i = p, \sum_{i=1}^r q_i = q\right\}.
    \]
    For each $\underline{d} \in \calD(\psi)$, we will attach a cohomologically induced representation $\scrA_{\underline{d}}(\psi) = A_{\frakq(x_{\underline{d}})}(\lambda_{\underline{d}})$, see \S \ref{def_A_d_psi}.
    Then, the set $\Pi(\psi)$ defined by
    \[
    \Pi(\psi) = \{\scrA_{\underline{d}}(\psi) \mid \underline{d} \in \calD(\psi)\}
    \]
    is the Arthur packet or $A$-packet associated with $\psi$ by \cite[Th\'eor\`eme 1.1]{MR_consequence}.

    For a good $A$-parameter $\psi = \bigoplus_{i=1}^r \chi_{t_i}\otimes S_{a_i}$, let $j = j(\psi)$ be the minimal number $i$ so that $\sum_{\ell=1}^i a_\ell \geq p$.
    Put $a_{<i} = \sum_{\ell<i} a_\ell$ and $a_{>i} = \sum_{\ell>i} a_\ell$.
    Define $\underline{d_0} \in \calD(\psi)$ by
    \[
    \underline{d_0} = \underline{d_0}(\psi) = \{(a_1,0),\ldots,(a_{j-1},0), (p_j,q_j), (0,a_{j+1}),\ldots, (0,a_{r})\}
    \]
    where $p_j = p - a_{<j}, q_j = q - a_{>j}$.
    Let $\nu_i$ be the segment $[(t_i-a_i+1)/2,(t_i+a_i-1)/2]$ and 
    \[
    \nu_{<i} = \bigsqcup_{k < i} \nu_k, \qquad \nu_{>i} = \bigsqcup_{i < k} \nu_k.
    \]
    Here, we consider the union as multisets.
    The multisets $\nu_{\leq i}$ and $\nu_{\geq i}$ are defined similarly.
    The representations in $\Pi(\psi)$ have the infinitesimal character $\chi_\psi$ which has the Harish-Chandra parameter $\nu = \sqcup_i \nu_i$.

    For an irreducible representation $\lambda = (\lambda_1,\ldots,\lambda_N)$ of $K$, put
    \[
    p' = p'(\lambda) = \#\{i \mid \lambda_p = \lambda_i, 1 \leq i \leq p\}, \qquad q' = q'(\lambda) = \#\{i \mid \lambda_i = \lambda_{p+1}, p+1 \leq i \leq N\}.
    \]
    Set
    \[
    P = P(\lambda) = \{\lambda_{p}-(N-1)/2,\lambda_{p-1}-(N-1)/2+1\ldots,\lambda_1+(p-q-1)/2\}
    \]
    and
    \[
    Q = Q(\lambda) = \{\lambda_{N}+(p-q+1)/2, \lambda_{N-1}+(p-q+3)/2,\ldots, \lambda_{p+1}+(N-1)/2\}.
    \]
    The multiset $P \sqcup Q$ can be identified with the infinitesimal character of the lowest weight representation $\pi_\lambda$.
    We define the segments $P'$ and $Q'$ by
    \[
    P' = [\lambda_p-(N-1)/2, \lambda_p-(N+1)/2+p'], \qquad Q' = [\lambda_{p+1}+(N+1)/2-q',\lambda_{p+1}+(N-1)/2].
    \]
    Put $I = P' \cap Q'$.

    \begin{lem}[Lemma \ref{lemma}]
        Let $\lambda = (\lambda_1,\ldots,\lambda_{N})$ be an irreducible representation of $K$ and $\pi_\lambda$ be the irreducible lowest weight representation with lowest $K$-type $\lambda$.
        For an $A$-parameter $\psi$, if $\pi_\lambda \in \Pi(\psi)$, the parameter $\psi$ is good and $\chi_\psi = \chi_{\pi_\lambda}$, i.e., $\nu = P \sqcup Q$.
        Moreover, if $\scrA_{\underline{d}}(\psi) \isom \pi_\lambda$, we have $\underline{d} = \underline{d_0}$.
    \end{lem}
    
    As a consequence of this lemma, there exists at most one unitary lowest weight representation in $\Pi(\psi)$.
    We now state the main theorem of the present paper.
    
    \begin{thm}[Theorem \ref{main}]
        Let $\lambda = (\lambda_1,\ldots,\lambda_{N})$ be an irreducible representation of $K$ and $\pi_\lambda$ be the irreducible lowest weight representation with lowest $K$-type $\lambda$.
        Suppose that $\scrA_{\underline{d_0}}(\psi)$ is nonzero and $\chi_\psi = \chi_{\pi_\lambda}$.
            
        \begin{enumerate}
            \item If $N- p' \leq \lambda_{p}-\lambda_{p+1} < N-q'$, the packet $\Pi(\psi)$ contains $\pi_\lambda$ if and only if $[\lambda_p-(N-1)/2, \lambda_{p+1}+(N-1)/2] \subset \nu_j \subset P'$.
            \item If $N-q' \leq \lambda_{p}-\lambda_{p+1} < N-p'$, the packet $\Pi(\psi)$ contains $\pi_\lambda$ if and only if either
            \begin{itemize}
                \item $\nu_{\leq j} = P$, or
                \item $[\lambda_{p}-(N-1)/2, \lambda_{p+1}+(N-1)/2] \subset \nu_j \subset Q'$.
            \end{itemize}
            \item If $N- p', N-q' \leq \lambda_{p}-\lambda_{p+1}$, the packet $\Pi(\psi)$ contains $\pi_\lambda$ if and only if either
            \begin{itemize}
                \item $P \subset \nu_{\leq j} \subset P \sqcup I$, or
                \item $I \subset \nu_{j} \subset Q'$.
            \end{itemize}
            \item If $\lambda_{p}-\lambda_{p+1} < N- p', N-q'$, the packet $\Pi(\psi)$ contains $\pi_\lambda$ if and only if $[\lambda_p-(N-1)/2, \lambda_{p+1}+(N-1)/2] = \nu_j$.
        \end{enumerate}
    \end{thm}

    Conversely, we have the following:
    \begin{thm}[Corollary \ref{psi_to_lambda}]
        Let $\psi = \bigoplus_{i=1}^r \chi_{t_i} \otimes S_{a_i}$ be a good $A$-parameter.
        The representation $\Pi(\psi)$ contains a non-zero unitary lowest weight representation if and only if  both
        \begin{itemize}
            \item $\nu_{<j}$ and $\nu_{>j}$ are multiplicity free, and
            \item $\#(\nu_j \cap \nu_{>j}) \leq p_j$ and $\#(\nu_j \cap \nu_{<j}) \leq q_j$.
        \end{itemize}
        When $\Pi(\psi)$ contains a nonzero unitary lowest weight representation $\pi$ in $\Pi(\psi)$, the lowest $K$-type $\lambda$ of $\pi$ is given as follows:
        \begin{enumerate}
            \item When $q_j = 0$, the lowest $K$-type $\lambda$ of $\pi$ satisfies $P(\lambda) = \nu_{\leq j}$ and $Q(\lambda) = \nu_{>j}$.
            \item When $p_j = \#(\nu_j \cap \nu_{>j})$ and $q_j \neq 0$, the lowest $K$-type $\lambda$ of $\pi$ satisfies $P(\lambda) = \nu_{<j} \sqcup (\nu_{j} \cap \nu_{>j})$ and $Q(\lambda) = \nu_{\geq j} \setminus  (\nu_{j} \cap \nu_{>j})$.
            \item When $q_j = \#(\nu_j \cap \nu_{<j}) \neq 0$, the lowest $K$-type $\lambda$ of $\pi$ satisfies $P(\lambda) = \nu_{\leq j}\setminus (\nu_{<j} \cap \nu_j)$ and $Q(\lambda) = (\nu_j \cap \nu_{<j}) \sqcup \nu_{>j}$.
            \item When $p_j \neq \#(\nu_j \cap \nu_{>j})$ and $q_j \neq \#(\nu_j \cap \nu_{<j})$, put $\nu_{<j} \sqcup \nu_{>j} = \{\sigma_1,\ldots,\sigma_{N-\#(\nu_j)}\}$.
            Let $i_0$ be the minimal integer such that $1 \leq i_0 \leq N - \#(\nu_j)$ and $\#(\nu_j) - i_0 + 1 + \#\{x \in\nu_{<j} \sqcup \nu_{>j} \mid x > \nu_{j,i_0}\}  = p$.
            Then, the lowest $K$-type $\lambda = (\lambda_1,\ldots,\lambda_N)$ of $\pi$ is given by
            \begin{align*}
            \lambda_i
            =
            \begin{cases}
                \sigma_i -(p-q+1)/2+i & \text{if $i < p-\#(\nu_j)+i_0$.}\\
                \nu_{j, 1} + (N+1)/2 -\#(\nu_j) & \text{if $p-\#(\nu_j)+i_0 \leq i \leq p$.}\\
                \nu_{j,1} - (N-1)/2 &\text{if p+1 $\leq i \leq p+i_0-1$.}\\
                \sigma_{i-\#(\nu_j)}-(N+1)/2-p+i & \text{if $p + i_0 \leq i$.}
            \end{cases}
            \end{align*}
        \end{enumerate}
    \end{thm}
    
    To conclude the introduction, we give some remarks and possible applications of the present paper.
    In the proof of the main theorem, we do not calculate the $K$-types of $A_{\frakq}(\lambda)$ except for special cases.
    Our proof is based on the Barbasch-Vogan parametrization of representations of $G$.
    This parametrization says that for any irreducible representation $\pi$, the map $\pi \mapsto (\Ann(\pi), \mathrm{AS}(\pi))$ is injective, where $\Ann(\pi)$ is the annihilator and $\mathrm{AS}(\pi)$ is the asymptotic support.
    The invariants $\Ann(\pi)$ and $\mathrm{AS}(\pi)$ can be described as certain tableaux in our case.
    Trapa \cite{Trapa_2001} gave an algorithm to compute such invariants for cohomological inductions $A_\frakq(\lambda)$.
    We calculate the tableaux and investigate the conditions where the cohomological inductions $A_{\frakq}(\lambda)$ are isomorphic to a given $\pi_\lambda$.
    When the group $G$ is not unitary, the problem becomes complicated, and there are at least two difficulties that do not occur in the unitary case: one, the description of unipotent $A$-parameters and two, the reducibility of $A_\frakq(\lambda)$ in the weakly fair range.
    
    Finally, the results of this paper have been applied to the geometry of Shimura varieties of $\mathrm{U}(1, n)$ in \cite{HMY_2024} by the author and the collaborators. 
    Beyond this specific application, we anticipate that the result will be utilized in various fields, leading to numerous further developments.

\subsection*{Acknowledgement}
    We wish to thank Takuya Yamauchi for invaluable comments and encouragement.
    We would also like to thank Du Chengyu for sharing ideas on the proof of the main theorem and clarifying our questions about the properties of $A_\frakq(\lambda)$.
    The author is partially supported by JSPS KAKENHI Grant Number 23K12965.

\section{Unitary groups and representations}
    In this section, we review definitions of unitary groups and representations.
    We also recall the Barbasch-Vogan parametrization of the representations, which is a key tool in this study.

    \subsection{Tableau notation}
    For a segment, we mean a finite set of the form $\{a, a+1, \ldots, a+n\}$, say $[a, a+n]$, for a real number $a$ and $n \in \bbZ_{\geq 0}$.
    In this paper, the segments are always regarded as multisets, that is, sets with multiplicities.
    For segments $\nu_{1} = [a,b]$ and $\nu_2 = [c,d]$, we say that $\nu_1$ and $\nu_2$ are linked (resp.~$\nu_1 \leq \nu_2$) if either $c=b+1$ or $a = d+1$ (resp.~$a \leq c$ and $b \leq d$) holds.

    For a partition $n = n_1 + \cdots + n_\ell$ with $n_1 \geq \cdots \geq n_\ell$, we have a diagram with $n_i$ boxes in the $i$-th row.
    This diagram is called a Young diagram of size $n = n_1 + \cdots + n_\ell$.
    If $\nu = (\nu_1,\ldots,\nu_n)$ is an $n$-tuple of real numbers, a $\nu$-quasitableau is defined as a tableau such that the shape is the Young diagram of size $n$ and the entries are an arrangement of $\nu_1,\ldots,\nu_n$.
    For a $\nu$-quasitableau $T$, we say that $T$ is $\nu$-antitableau if entries strictly decrease down each column and weakly decrease along each row.
    This definition is the same as the definition of the semistandard tableau by replacing ``decreasing" with ``increasing."

    We define a $(p,q)$-signed tableau as an equivalence class of Young diagrams whose boxes are $p$ plus boxes and $q$ minus boxes so that the signs alternate across the row.
    Here, we say that two signed tableaux are equivalent if the signatures are the same and coincide by interchanging rows of the same length.
    For a tableau $T$ with entries, we say that the $(a,b)$-th entry of $T$ is the entry in the $a$-th row and the $b$-th column.
    
    The definitions above are those in the previous works \cite{C-Huang_2024, Chengyu_2024, Trapa_2001} on the nonvanishing of cohomological inductions.
    See these references for examples of tableaux.
    
\subsection{Unitary groups} 
    For a Lie group $H$, we denote by $\frakh$ (resp.~$\frakh_\bbC$) the Lie algebra of $H$ (resp.~complexification of $\frakh$) and by $\calU(\frakh_\bbC)$ the universal enveloping algebra of $\frakh_\bbC$.
    Fix a positive integer $N$ with a partition $N = p+q$ and $p,q \geq 0$.
    We define the unitary group $G = \U(p,q)$ by
    \[
    G = \U(p,q) = 
    \left\{
    g \in \GL_{N}(\bbC) 
    \,\middle|\,
    {^t\overline{g}}
    I_{p,q}
    g
    = 
        I_{p,q}
    \right\},\qquad 
    I_{p,q} = \begin{pmatrix}
        \mathbf{1}_p&\\
        &-\mathbf{1}_q
    \end{pmatrix}.
    \]
    Here, $\overline{g}$ is the complex conjugate of $g$.
    For a Cartan involution $\theta \colon g \mapsto \Ad(I_{p,q})({^t\overline{g}^{-1}})$, let $K$ be the group of the fixed points of $\theta$.
    Then, $K$ is a maximal compact subgroup of $G$, which is isomorphic to $\U(p) \times \U(q)$.
    The Cartan involution $\theta$ induces an involution $\theta$ on $\frakg$ and a decomposition
    \[
    \frakg = \frakk \oplus \frakp, \qquad 
    \frakg_\bbC = \frakk_\bbC \oplus \frakp_+ \oplus \frakp_-.
    \]
    Here, $\frakp$ is the $(-1)$-eigenspace of $\theta$ on $\frakg$, and $\frakp_+$ (resp.~$\frakp_-$) corresponds to the holomorphic (resp.~anti-holomorphic) tangent space of a hermitian symmetric space $G/K$.
    Let $T$ be the diagonal subgroup of $G$.
    Then, $T$ is a Cartan subgroup of $G$ and is stable under the Cartan involution $\theta$.
    Define $e_i \in \frakt_\bbC^*$ by $e_i(\diag(t_1,\ldots,t_N)) = t_i$.
    We regard $\frakt_\bbC^*$ as $\bbC^N$ by the basis $\{e_1,\ldots,e_N\}$.
    Then, the root system $\triangle$ of $\frakg_\bbC$ with respect to $\frakt_\bbC$ is equal to
    \[
    \triangle = \{\pm(e_i - e_j) \mid 1 \leq i < j < N\}.
    \]
    We choose a positive system $\triangle^+$ as
    \[
    \triangle^+ = \{e_i - e_j \mid 1 \leq i < j \leq N\}.
    \]
    Let $\triangle_c$ (resp.~$\triangle_n$) be the compact (resp.~noncompact) root system with the positive system $\triangle^+_c = \triangle_c \cap \triangle^+$ (resp.~$\triangle^+_n = \triangle_n \cap \triangle^+$), i.e.,
    \[
    \triangle_c = \{\pm(e_i-e_j) \mid \text{$1 \leq i < j \leq p$ or $p+1 \leq i < j \leq N$}\}
    \]
    and 
    \[
    \triangle_n = \{\pm(e_i - e_j) \mid 1 \leq i \leq p <  j \leq N\}.
    \]
    Then, the root system of $\frakp_+$ associated with $\frakt_\bbC$ is $\triangle_n^+$.
    Let $\frakb$ be the Borel subalgebra of $\frakg_\bbC$ associated with $\triangle^+$ and $\frakb^-$ be the opposite of $\frakb$.
    Let $\fraku$ be a Lie subalgebra of $\frakg_\bbC$ that is stable under the adjoint action of $\frakt_\bbC$.
    Then, let $\rho(\fraku)$ be half the sum of roots in $\fraku$ associated with $\frakt_\bbC$.
    Put $\rho = \rho(\frakb)$.

    For $\lambda = (\lambda_1,\ldots,\lambda_N) \in \bbZ^N$ with $\lambda_1 \geq \cdots \geq \lambda_p, \lambda_{p+1} \geq \cdots \geq \lambda_{N}$, let $F(\lambda)$ be an irreducible representation of $K$ with the highest weight $\lambda$.
    We often write $F(\lambda)$ as $\lambda$, for short.
    The restriction of an admissible $(\frakg, K)$-module $\pi$ to $K$ can be decomposed as a direct sum 
    \[
    \pi|_K = \bigoplus_{\lambda} F(\lambda)^{\oplus m_\pi(\lambda)}, \qquad  m_\pi(\lambda) \in \bbZ_{\geq 0}
    \]
    where $\lambda$ runs over all $\triangle^+_c$-dominant integral weights.
    The nonnegative integer $m_\pi(\lambda)$ is the multiplicity of $F(\lambda)$ in $\pi$.
    For a $K$-type of $\pi$, we mean $F(\lambda)$ with $m_\pi(\lambda) \neq 0$.

    For an infinitesimal character $\chi_\pi$ of an irreducible representation $\pi$, we mean a character of the center $\calZ(\frakg_\bbC)$ of $\calU(\frakg_\bbC)$ obtained by the restriction of $\pi$ to $\calZ(\frakg_\bbC)$.
    By the Harish-Chandra isomorphism, infinitesimal characters are parametrized by $\frakt^*_\bbC/W$, where $W = W(G; T)$ is the Weyl group of $G$ for $T$.
    For $\lambda \in \frakt^*_\bbC/W$, let $\chi_\lambda$ (or $\lambda$ for short) be the corresponding character of $\calZ(\frakg_\bbC)$.
    We say that an infinitesimal character $\nu$ is integral if $\nu$ is in the image of $\bbZ^N + \rho$.
    In this paper, we regard integral infinitesimal characters as multisets with $N$-elements or an element $(x_1,\ldots,x_N)$ in $\bbZ^N + (N-1)/2$ with $x_1 \geq \cdots \geq x_N$.

    \subsection{Unitary lowest weight representations}
    For a $(\frakg, K)$-module $\pi$, we say that $\pi$ is lowest weight if there exists $v \in \pi$ such that $v$ generates $\pi$ and $v$ is annihilated by $\frakb^-$.
    Let $\lambda = (\lambda_1,\ldots,\lambda_N) \in \bbZ^N \subset \frakt_\bbC^*$ be a $\triangle_c^+$-dominant integral weight.
    We regard the irreducible representation $F(\lambda)$ of $K$ as an irreducible $\frakp_- \oplus \frakk_\bbC$ module by letting $\frakp_-$ act as trivial.
    Set
    \[
    N(\lambda) = \calU(\frakg_\bbC) \otimes_{\calU(\frakp_- \oplus \frakk_\bbC)}  F(\lambda).
    \]
    The module $N(\lambda)$ is called the parabolic Verma module and has the unique irreducible quotient $L(\lambda)$ (cf.~\cite[\S 9.4]{cat_o}).
    The module $L(\lambda)$ is unitarizable if $\lambda_p - \lambda_{p+1} \geq N -p' - q'$ where $p' = p'(\lambda) = \#\{i \mid \lambda_i = \lambda_p, i \leq p\}$ and $q' = q'(\lambda) = \#\{i \mid \lambda_i = \lambda_{p+1}, p+1 \leq i \leq N\}$ by \cite[Theorem 2.4]{EHW_83}.
    Note that $L(\lambda)$ is a discrete series representation if $\lambda_{p} - \lambda_{p+1} > N-1$ and is limits of discrete series if $\lambda_p - \lambda_{p+1} = N-1$.
    The infinitesimal character of $\pi_\lambda$ equals
    \[
    \left(\lambda_1+\frac{p-q-1}{2}, \ldots, \lambda_{p}-\frac{N-1}{2}, \lambda_{p+1}+\frac{N-1}{2}, \ldots, \lambda_N + \frac{p-q+1}{2}\right).
    \]
    Note that the infinitesimal character $\chi_{\pi_\lambda}$ is integral by $\lambda \in \bbZ^N$.

\subsection{Barbasch-Vogan parametrization of representations of \texorpdfstring{{$G$}}{}}\label{BV-param}

    For the details of this subsection, we refer to \cite[\S 4,5]{Trapa_2001} and \cite{Barbasch-Vogan_1983_weyl}.
    We introduce two invariants $\Ann(\pi)$ and $\mathrm{AS}(\pi)$ from $(\frakg, K)$-modules $\pi$, which are complete invariants of $\pi$. 
    For a $(\frakg, K)$-module $\pi$, let $\Ann(\pi)$ be the annihilator of $\pi$ in $\calU(\frakg_\bbC)$.
    When $\pi$ is irreducible, the ideal $\Ann(\pi)$ is called a primitive ideal.
    Primitive ideals for an irreducible representation with infinitesimal character $\nu$ are parametrized by $\nu$-antitableau, where $\nu$ is the infinitesimal character of $\pi$.
    From the local behavior of the character of $\pi$, we may define the asymptotic support $\mathrm{AS}(\pi)$ of $\pi$.
    Note that asymptotic supports are the union of nilpotent orbits.
    
    \begin{thm}[{\cite[Theorem 4.2]{Barbasch-Vogan_1983_weyl}, \cite[Theorem 6.1]{Trapa_2001}}]
        For irreducible $(\frakg, K)$-modules $\pi$ and $\pi'$ with integral infinitesimal characters, the representation $\pi$ is isomorphic to $\pi'$ if and only if $(\Ann(\pi), \mathrm{AS}(\pi)) = (\Ann(\pi'), \mathrm{AS}(\pi'))$.
    \end{thm}
    
    For cohomological inductions, which we will treat in this paper, the asymptotic support $\mathrm{AS}(\pi)$ is a single unipotent orbit by \cite[Proposition 5.4]{Trapa_2001}.
    Then, the asymptotic support $\mathrm{AS}(\pi)$ can be parametrized by $(p,q)$-signed tableaux (cf.~\cite[Theorem 9.3.3]{Collingwood-McGovern_1993}).
    Hence, representations treated in this paper will be parametrized by a $\nu$-antitableau and a $(p, q)$-signed tableau.

    Trapa and Vogan conjectured (\cite[Conjecture 1.1]{Trapa_2001}) that the cohomologically induced representations in the weakly fair range exhaust the unitary $(\frakg, K)$-modules with integral infinitesimal characters.
    As far as we know, the conjecture has been proven only for specific cases such as $\U(n,1)$ and $\U(n,2)$ (see \cite{Wong-Zhang_2024_unitary_dual_u(n_2)} for details).

\section{Cohomologically induced representations}
    In this section, we recall the cohomologically induced representations $A_{\frakq}(\lambda)$ and their basic properties.
    We also recall Trapa's algorithm to determine the tableaux for $A_\frakq(\lambda)$ and give a nonvanishing criterion for certain $A_{\frakq}(\lambda)$.
    
\subsection{\texorpdfstring{$\theta$}{}-stable parabolic subalgebras and cohomologically induced representation}

    Take $x \in \sqrt{-1}\,\frakk$.
    Since the action $\mathrm{ad}(x)$ on $\frakg_\bbC$ is diagonalizable with real eigenvalues, we may define the subalgebras of $\frakg_\bbC$ by
\begin{align*}
    \frakq=\frakq(x) 
    &= \text{sum of root vectors with nonnegative eigenvalues}.\\
    \fraku =\fraku(x)
    &= \text{sum of root vectors with positive eigenvalues}.\\
    \frakl=\frakl(x) 
    &= \text{sum of root vectors with zero eigenvalues}.
\end{align*}
    We call parabolic subalgebras $\frakq$ of $\frakg_\bbC$ obtained in the above way $\theta$-stable parabolic subalgebras of $\frakg_\bbC$.
    We may assume $x \in \sqrt{-1}\,\frakt$ by a conjugate of an element in $K$.
    For $\lambda \in \frakt_\bbC^*$, let $\bbC_\lambda$ denote the character of $\frakl$ with $\bbC_\lambda |_{\frakt} = \lambda$, if it exists.
    We then obtain the cohomologically induced representation $A_\frakq(\lambda)$ by the induction of $\bbC_\lambda$ as in \cite[(5.6)]{Knapp-Vogan}.

    In our case, the $\theta$-stable parabolic subalgebra $\frakq(x)$ will be described as follows:
    Let $\underline{d} = \{(p_i, q_i)_{1 \leq i \leq r}\}$ and $x_{\underline{d}}$ be an element in $\sqrt{-1}\,\frakt \isom \bbR^N$ of the form
    \[
    x_{\underline{d}} = (\underbrace{x_1, \ldots,x_1}_{p_1}, \ldots,\underbrace{x_r,\ldots,x_r}_{p_r},\underbrace{x_1, \ldots,x_1}_{q_1}, \ldots,\underbrace{x_r,\ldots,x_r}_{q_r}), \qquad x_1 > \cdots > x_r
    \]
    such that $(p_i,q_i) \in (\bbZ_{\geq 0})^2$ with $p_iq_i \neq 0$ for any $i$ and $p_1 + \cdots + p_r = p, q_1 + \cdots + q_r = q$.
    Set $\frakq_{\underline{d}} = \frakq(x_{\underline{d}})$.
    The centralizer $C_G(\frakq_{\underline{d}})$ is a connected reductive group $L_{\underline{d}}$ that is isomorphic to 
    \[
    L_{\underline{d}} \isom \mathrm{U}(p_1,q_1) \times \cdots \times \mathrm{U}(p_r,q_r)
    \]
    such that $\frakl_{\underline{d}} = \frakl(x_{\underline{d}})_\bbC = \Lie(L_{\underline{d}}) \otimes \bbC$.
    Note that our choice of $\frakq_{\underline{d}}$ or the hermitian form $I_{p,q}$ is different from that of \cite{MR_consequence, Trapa_2001, Vogan_1997_cohomology_and_group_representations}, but the same as \cite{Ichino_2022}.
   We say that a $\theta$-stable parabolic subalgebra $\frakq$ is holomorphic if there exists $j$ such that $q_i = 0$ for any $i < j$ and $p_\ell = 0$ for any $\ell > j$, i.e., $\fraku \cap \frakp_\bbC \subset \frakp_+$.
    The choices for a holomorphic $\theta$-stable parabolic subalgebra are the same as \cite{Trapa_2001, Vogan_1997_cohomology_and_group_representations}.

\subsection{Properties of \texorpdfstring{$A_\frakq(\lambda)$}{}}

    In this subsection, we review the basic properties for $A_\frakq(\lambda)$.
        Let $\underline{d} = \{(p_i, q_i)_{1 \leq i \leq r}\}$ be a pair of nonnegative integers with $(p_i, q_i) \neq 0, \sum_i p_i = p$, and $\sum_i q_i = q$.
        Take $\lambda \in \bbZ^N$ such that
        \[
        \lambda = (\underbrace{\lambda_1,\ldots,\lambda_1}_{a_1},\ldots,\underbrace{\lambda_r,\ldots,\lambda_r}_{a_r}), \qquad a_i = p_i+q_i.
        \]
        We define the segments $\nu_i$ associated with $\underline{d}$ and $\lambda$ by
        \[
        \nu_i = \nu_i(\underline{d}, \lambda) = [\lambda_{i}+(N+1)/2-a_{\leq i}, \lambda_{i}+(N-1)/2-a_{<i}].
        \]
        Here, $a_{<i} = \sum_{\ell < i} a_\ell$ and $a_{\leq i} = \sum_{\ell \leq i} a_\ell$.
        We sometimes denote the cohomological induction $A_{\frakq_{\underline{d}}}(\lambda)$ by $A(\frakq_{\underline{d}}, \nu_1,\ldots,\nu_r) = A(\underline{d}, \nu_1,\ldots,\nu_r)$ for short.
        We say that the cohomological induction $A_{\frakq_{\underline{d}}}(\lambda)$ is in the weakly fair range (resp.~mediocre range) if $\lambda_{i} - \lambda_{i+1} \geq -(a_i+a_{i+1})$ for any $i$ (resp.~$\lambda_{i}-\lambda_{j}\geq -\max\{n_i,n_j\} - \sum_{i < k < j} a_k$ for any $i < j$).
        We also say that $\nu = \lambda+\rho = \sqcup_i \nu_i$ is in the weakly fair range (resp.~mediocre range) if $A_{\frakq_{\underline{d}}}(\lambda)$ is so.

    Note that the cohomological induction $A_{\frakq_{\underline{d}}}(\lambda)$ is in the weakly fair range (resp.~mediocre range) if and only if the mean value in $\nu_i$ is greater than or equal to the mean value in $\nu_{i+1}$ (resp.~$\nu_i \not\geq \nu_j$ for any $i < j$) by the explicit calculation (cf.~\cite[Lemma 2.4]{Chengyu_2024}).

    The following statements are well-known (for example, see \cite{Knapp-Vogan}, \cite[Lemma 4.2]{1987_Adams}, \cite[Theorem 6.4.4]{Huang-Pandzic_2007_dirac_op_rep_theory}, and \cite[Theorem 3.1 (iv)]{Trapa_2001}):
    \begin{itemize}
        \item In the weakly fair range, $A_\frakq(\lambda)$ is unitarizable.
        Moreover, since $G$ is a real unitary group, it is zero or irreducible.
        \item For $A_\frakq(\lambda)$ in the mediocre range, the formula
        \[
        \dim \Hom_K(F(\mu), A_{\frakq}(\lambda)) = \sum_{w \in W^1} \sgn(w) P_{\fraku \cap \frakp_\bbC} (w(\mu+\rho_c) - (\lambda+2\rho(\fraku\cap \frakp) + \rho_c))
        \]
        holds.
        Here, $W^1$ is the subgroup of the Weyl group of $K$ with respect to $T$ consisting of $w$ for which $\alpha \in \triangle^+_c, w^{-1}\alpha < 0$ implies $\alpha \in \triangle(\fraku \cap \frakp_\bbC)$, and $P_{\fraku \cap \frakp_\bbC}$ is the partition function with respect to $\fraku \cap \frakp_\bbC$, i.e., $P_{\fraku \cap \frakp_\bbC}(x)$ the multiplicity of weight $x$ in the symmetric algebra $S(\fraku \cap \frakp_\bbC)^{\frakl \cap \frakk \cap \frakn}$.
        Moreover, if $\lambda+2\rho(\fraku\cap \frakp_\bbC)$ is $\triangle_c^+$-dominant, $\lambda + 2\rho(\fraku \cap \frakp_\bbC)$ occurs in $A_\frakq(\lambda)$ and  $K$-types in $A_\frakq(\lambda)$ are of the form
        \begin{align}\label{K-type_formula}
        \lambda+2\rho(\fraku\cap \frakp_\bbC) + \sum_{\alpha \in \triangle(\fraku \cap \frakp_\bbC)} n_\alpha \alpha, \qquad n_\alpha \geq 0.
        \end{align}
    \end{itemize}

    Calculating the $K$-type formula directly seems complicated, even in the case where $\frakq$ is a holomorphic $\theta$-stable parabolic subalgebra.
    To avoid such complexity, we will instead calculate the associated tableaux described below.

\subsection{Tableaux associated with \texorpdfstring{$A_\frakq(\lambda)$}{}}\label{trapa_algorithm}

    Let $\pi$ be a cohomological induction in the weakly fair range.
    We associated two invariants $\Ann(\pi)$ and $\mathrm{AS}(\pi)$ to $(\frakg, K)$-modules in section \ref{BV-param}.
    The annihilator $\Ann(\pi)$ can be regarded as a $\nu$-antitableau, but $\mathrm{AS}(\pi)$ is a union of unipotent orbits.
    If $\pi$ is isomorphic to a cohomological induction $A_\frakq(\lambda)$ in the weakly fair range, the asymptotic support $\mathrm{AS}(\pi)$ is a single unipotent orbit by \cite[Proposition 5.4]{Trapa_2001}.
    Indeed, for $A_\frakq(\lambda)$ in the good range, its asymptotic support is a single unipotent orbit.
    Since we can obtain $\pi$ as a translation of cohomological inductions in the good range, one can show that the asymptotic supports coincide.
    Hence, $\mathrm{AS}(A_{\frakq}(\lambda))$ is a single unipotent orbit.
    Thus, we may associate a $\nu$-antitableau and a $(p,q)$-signed tableau for a cohomological induction in the weakly fair range.
    The examples of tableaux associated with $A_\frakq(\lambda)$ are available in \cite{C-Huang_2024, Chengyu_2024, Trapa_2001}.

    In the following, we associate tableaux for $A_\frakq(\lambda)$.
    Let $\frakq$ be the $\theta$-stable parabolic subalgebra corresponding to $\{(p_i, q_i)_{1 \leq i \leq r}\}$ and $\lambda$ be a $\triangle_c^+$-dominant integral weight such that $A_\frakq(\lambda)$ is in the mediocre range.
    We first construct the $(p,q)$-signed tableau inductively.
    Let $S_1$ be the Young tableau of size $1+ \cdots +1$ with $p_1+q_1$ boxes filled with $p_1$ plus boxes and $q_1$ minus boxes.
    Suppose that the $(\sum_{i<k}p_i, \sum_{i<k}q_i)$-signed tableau $\sqcup_{i<k} S_k$ for $k \leq r$ is defined.
    We then construct the signed tableau $\sqcup_{i \leq k} S_k$ by adding $p_k$ boxes filled with $+$ and $q_k$ boxes filled with $-$, from top to bottom, to each row-ends of $\sqcup_{i < k} S_k$ such that
    \begin{itemize}
        \item at most one box is added in each row end, and
        \item the signs in $\sqcup_{i \leq k} S_k$ are alternating across the row.
    \end{itemize}

    We next construct the $\nu$-antitableau.
    For $A_\frakq(\lambda)$, the shape of $(p,q)$-signed tableau and the $\nu$-antitableau are the same.
    The shape of the $(p,q)$-signed tableau is partitioned into $\sqcup_{1 \leq i \leq r} S_i$, which is the same as in the definition of the $(p,q)$-signed tableau.
    For each $S_i$, we fill $\nu_{i,1},\ldots,\nu_{i,a_i}$ from top to bottom, where $\nu_i = \nu_i(\underline{d}, \lambda) = \{\nu_{i,1},\ldots,\nu_{i,a_i}\}$ with $\nu_{i,1}> \cdots > \nu_{i,a_i}$.
    When the resulting tableau $S$ is a $\nu$-antitableau, let $\Ann(A_{\frakq}(\lambda)) = S$, which is possibly equivalent to the formal zero tableau explained below.
    We now define the two invariants $\mathrm{overline}(S_i, S_{i+1})$ and $\mathrm{sing}(S_i, S_{i+1})$ associated with cohomological inductions.
    Set
    \[
    \mathrm{sing}(S_i, S_{i+1}) = \#(\nu_{i} \cap \nu_{i+1}).
    \]
    For $S_i$ and $S_{i+1}$ as in the definition of $(p,q)$-signed tableau $S = \sqcup_i S_i$, let $m$ be the largest integer such that for any $i$ with $i \leq m$, the $(a_i - m + i)$-th box of $S_i$ is strictly left to the $i$-th box of $S_{i+1}$.
    If such an integer $m$ does not exist, set $m=0$.
    We then define $\mathrm{overlap}(S_i, S_{i+1})$ by the nonnegative integer $m$.
    In the following, we give an algorithm \cite[Procedure 7.5]{Trapa_2001} to obtain a $\nu$-antitableau or the formal zero tableau from the tableau $S$.
    By \cite[Theorem 7.9]{Trapa_2001}, the cohomological induction $A_\frakq(\lambda)$ is zero if and only if the tableau $S$ is equivalent to the formal zero tableau.
    More precisely, let $S' = \sqcup_i S_i'$ be the tableau after Trapa's algorithm.
    Let $\nu_i'$ be the segment consisting of the entries of $S_i'$.
    The cohomological induction $A_\frakq(\lambda)$ is nonzero if and only if the resulting tableau $S' = \sqcup_i S'_i$ is not the zero tableau and satisfies
    \begin{itemize}
        \item $\nu_i \geq \nu_{i+1}$, and
        \item $\mathrm{overlap}(S'_i, S'_{i+1}) \geq \mathrm{sing}(S_i', S_{i+1}')$ for any $i$.
    \end{itemize}

    Trapa's algorithm is the following.
    For $S = \sqcup_{i} S_i$, the algorithm is generated by replacing adjacent skew columns $S_i \sqcup S_{i+1}$ with $S_i' \sqcup S_{i+1}'$.
    Set $R = S_i \sqcup S_{i+1}$ and $R' = S_{i}' \sqcup S_{i+1}'$.
    When $R'$ is the zero tableau, we understand that the tableau $S$ is equivalent to the zero tableau.
    We review the construction of $R'$ that is an arrangement of $\nu_i \sqcup \nu_{i+1}$ with the shape $R$.
    \begin{enumerate}
        \item If one has $\mathrm{overlap}(S_i, S_{i+1}) = \mathrm{sing}(S_i, S_{i+1}) < \min\{a_i, a_{i+1}\}$ or $\mathrm{overlap}(S_i, S_{i+1}) > \mathrm{sing}(S_i,S_{i+1})$, set $R' = R$.
        Here, $a_i = \#(\nu_i)$.
        \item If one has $\mathrm{overlap}(S_i, S_{i+1}) < \mathrm{sing}(S_i, S_{i+1})$, then $R'$ is the zero tableau.
        \item Assume $\mathrm{overlap}(S_i, S_{i+1}) = \mathrm{sing}(S_i, S_{i+1}) = a_{i+1}$.
        In this case, we have $\nu_{i+1} \subset \nu_{i}$.
        We define $R'$ by induction on $m = \nu_{i+1,a_{i+1}} - \nu_{i,a_{i}}$.
        When $m=0$, set $R' = R$.
        When $m>0$, set $\nu_{i+1}(-) = \{\nu_{i+1,1}-1,\ldots,\nu_{i+1,a_{i+1}}-1\}$.
        We define $S_{i+1}(-)$ by the tableau with the shape $S_{i+1}$ filled with $\nu_{i+1}(-)$.
        Set $R(-) = S_{i} \sqcup S_{i+1}(-)$.
        By induction hypothesis, $R(-)'$ is defined.
        Then, in $R(-)'$, there exists at most one box $B$ filled with $\nu_{i+1, 1}-1$ and strictly to the right of the unique box filled with $\nu_{i+1,1}$ in $R(-)'$.
        If the box $B$ exists, add one to the entry in $B$.
        If no such box exists, add one to the entry in the left-most box filled with $\nu_{i+1,1}-1$ in $R(-)'$.
        We denote the resulting tableau by $R(-)_1'$.
        Now construct $R(-)_2'$ by the same procedure applied to $R(-)_1'$, but instead considering the entries $\nu_{i+1,2}$ and $\nu_{i+1,2}-2$.
        By the same procedure, we get $R(-)_{a_{i+1}}'$.
        Set $R' = R(-)_{a_{i+1}}'$.
        \item Assume $\mathrm{overlap}(S_i, S_{i+1}) = \mathrm{sing}(S_i, S_{i+1}) = a_{i}$.
        In this case $\nu_{i} \subset \nu_{i+1}$.
        We define $R'$ by induction on $m = \nu_{i+1,1} - \nu_{i,1}$.
        When $m=0$, set $R' = R$.
        When $m>0$, set $\nu_{i}(+) = \{\nu_{i,1}+1,\ldots,\nu_{i,a_{i}} + 1\}$.
        We define the tableau $S_i(+)$ by the tableau with the shape $S_i$ filled with $\nu_i(+)$.
        Set $R(+) = S_i(+) \sqcup S_{i+1}$.
        By induction hypothesis, $R(+)'$ is defined.
        Then, in $R(+)'$, there exists at most one box $B$ filled with $\nu_{i, a_{i}} + 1$ and strictly to the left of the unique box filled with $\nu_{i,a_i}$ in $R(+)'$.
        If the box $B$ exists, subtract one from the entry in the box $B$.
        If no such box exists, subtract one from the entry in the right-most box filled with $\nu_{i, a_{i}}+1$ in $R(+)'$.
        We denote the resulting tableau by $R(+)_1'$.
        Now construct $R(+)_2'$ by the same procedure applied to $R(+)_1'$, by the same procedure again, we get $R(+)_{a_{i}}'$.
        Set $R' = R(+)_{a_{i}}'$.
    \end{enumerate}

    We still need to consider a partition of the resulting tableau $R'$ into $R' = S_{i}' \sqcup S_{i+1}'$.
    The last box in $R'$ is the right-most box filled with $\nu_{i+1, a_{i+1}}' = \min\{\nu_{i,a_i}, \nu_{i+1, a_{i+1}}\}$.
    The box next to the last box is the right-most box filled with $\nu_{i+1, a_{i+1}} + 1$.
    This procedure stops when the entry of the box reaches $\min\{\nu_{k,1}, \nu_{k+1,1}\}$.
    Let ${S'}_i$ be the tableau of the remaining boxes.
    We then obtain the partition $R' = S'_{i} \sqcup S'_{i+1}$.
    
    If interested, the readers are recommended to calculate examples and check the well-definedness of the above definition.
    One can calculate the overlap explicitly by \cite[Theorem 3.6]{Chengyu_2024} and \cite[Lemma 5.4]{C-Huang_2024}.
    For the explicit description of entries in the tableau $R'$, see \cite[\S 4.5]{C-Huang_2024}.

\subsection{Unitary lowest weight representations as cohomological inductions}
    In this subsection, we determine a cohomological induction that is isomorphic to a given lowest weight representation $\pi_\lambda$.
    The following is one of the easiest cases in which to calculate the $K$-types.
    
\begin{lem}\label{q_j=0}
    Let $\frakq$ be a holomorphic $\theta$-stable parabolic subalgebra corresponding to $\underline{d} = \{(p_i, q_i)_i\}$ such that $p_iq_i = 0$ for any $i$ and $\pi = A_\frakq(\lambda)$ be a cohomological induction in the mediocre range.
    Put $j = \max\{i \mid q_i = 0\}$.
    Then, $\pi$ is a nonzero if and only if $\nu_{\leq j}$ and $\nu_{>j}$ are multiplicity free.
    If $\pi$ is nonzero, then $\pi$ is a unitary lowest weight representation with the lowest $K$-type
    \[
    \lambda + 2\rho(\fraku \cap \frakp_\bbC) = \lambda + (\underbrace{q,\ldots,q}_p, \underbrace{-p,\ldots,-p}_q).
    \]
    Moreover, let
    \[
    \sigma_i =
    \begin{cases}
        \lambda_i + q & \text{if $1 \leq i \leq p$.}\\
        \lambda_i - p &\text{if $p+1 \leq i \leq N$.}
    \end{cases}
    \]
    We denote by $i_0$ the maximal positive integer such that $\sigma_{p+\min\{p,q\}+1-i_0} \geq \sigma_{p+1-i_0}$, if it exists.
    If there is no such integer, put $i_0 = 0$.
    Then, the first column of the tableau $\Ann(A_{\frakq}(\lambda))$ consists of
    \[
    \sigma_{1},\sigma_2,\ldots,\sigma_{p-i_0}, \underbrace{\sigma_{p+\min\{p,q\}+1-i_0},\ldots, \sigma_{p+\min\{p,q\}}}_{i_0}, \sigma_{p+\min\{p,q\}+1},\ldots,\sigma_{N}
    \]
    and the second column consists of
    \[
    \sigma_{p+1},\ldots,\sigma_{p+\min\{p,q\}-i_0},\underbrace{\sigma_{p+1-i_0},\ldots,\sigma_{p}}_{i_0}.
    \]
\end{lem}
\begin{proof}
    The statement follows immediately from (\ref{K-type_formula}) and Trapa's algorithm.
\end{proof}

    We explicitly describe $\pi_\lambda$ in terms of cohomological inductions as follows:
\begin{lem}\label{A_q_lt_wt_rep}
    Let $\pi_\lambda$ be a unitary lowest weight representation with lowest $K$-type $\lambda = (\lambda_1,\ldots\lambda_N)$.
    We then have the following:
    \begin{enumerate}
        \item When $\lambda_p -\lambda_{p+1} < N-p', N-q'$, let $\frakq$ be the $\theta$-stable parabolic subgroup corresponding to
        \[
        \{\underbrace{(1,0),\ldots, (1,0)}_{p-p'}, (p', N-(\lambda_p-\lambda_{p+1})-p'), \underbrace{(0,1),\ldots,(0,1)}_{\lambda_p-\lambda_{p+1}-p+p'}\}
        \]
        and
        \[
        \mu = (\mu_1,\ldots,\mu_N)
        \]
        be an element in $\bbZ^N$ defined by
        \[
        \mu_i =
        \begin{cases}
            \lambda_i-q &\text{if $i \leq p-q'$}.\\
            \lambda_{p+1}+p-p' &\text{if $p-p' < i \leq N-(\lambda_{p}-\lambda_{p+1})-p'$}.\\
            \lambda_i + p &\text{if $N-(\lambda_{p}-\lambda_{p+1})-p' < i$}.
        \end{cases}
        \]
        Then, $A_{\frakq}(\mu) \isom \pi_\lambda$.
        \item When $\min\{N-p', N-q'\} \leq \lambda_p -\lambda_{p+1}$, let $\frakq$ be the $\theta$-stable parabolic subalgebra corresponding to
        \begin{align*}\label{q_q_j=0}
        \{\underbrace{(1,0),\ldots,(1,0)}_{p-p'}, (p',0),(0,q'),\underbrace{(0,1),\ldots, (0,1)}_{q-q'}\}
        \end{align*}
        and $\mu$ be an element in $\bbZ^N$ defined by
        \[
        \mu = (\underbrace{\lambda_1-q, \ldots, \lambda_p -q}_p,\underbrace{\lambda_{p+1} - p,\ldots,\lambda_N-p}_{q}).
        \]
        Then, $A_\frakq(\mu) \isom \pi_\lambda$.
        \item Let $\frakq$ be the $\theta$-stable parabolic subalgebra corresponding to
        \begin{align*}\label{q_q_j=0}
        \{\underbrace{(1,0),\ldots,(1,0)}_{p-p'}, (p',0),(0,q'),\underbrace{(0,1),\ldots, (0,1)}_{q-q'}\}
        \end{align*}
        and $\mu$ be an element in $\bbZ^N$ defined by $\mu = (\mu_1,\ldots,\mu_N)$ such that $\mu_1 \geq \cdots \geq \mu_p, \mu_{p+1} \geq \cdots \geq \mu_N,\mu_p = \mu_{p-1}= \cdots = \mu_{p-p'+1}$, and $\mu_{p+1} = \mu_{p+2} = \cdots = \mu_{p+q'}$.
        Suppose that $A_\frakq(\mu)$ is in the mediocre range.
        Then, $A_\frakq(\mu)$ is a nonzero lowest weight representation and the lowest $K$-type $\lambda = (\lambda_1,\ldots,\lambda_N)$ of $A_\frakq(\mu)$ satisfies $\min\{N - p'(\lambda), N - q'(\lambda)\} \leq \lambda_p - \lambda_{p+1}$, where $p' = \#\{i \mid 1 \leq i \leq p, \lambda_i = \lambda_p\}$ and $q' = \#\{i \mid p+1 \leq i \leq N, \lambda_i = \lambda_{p+1}\}$.
    \end{enumerate}
\end{lem}
\begin{proof}
    We calculate the $K$-types of $A_\frakq(\lambda)$.
    For (1), we have
    \begin{align*}
    2\rho(\fraku \cap \frakp_\bbC) 
    &= 
    (\underbrace{q,\ldots,q}_{p-p'}, \underbrace{\lambda_p-\lambda_{p+1}-p+p', \ldots, \lambda_p-\lambda_{p+1}-p+p'}_{p'},\\
    &\qquad \qquad \qquad \qquad\underbrace{-p+p',\ldots,-p+p'}_{N-(\lambda_{p}-\lambda_{p+1})-p')}, \underbrace{-p,\ldots,-p}_{\lambda_p-\lambda_{p+1}-p+p'})
    \end{align*}
    and
    \[
    \mu + 2\rho(\fraku \cap \frakp) = (\lambda_1,\ldots,\lambda_p, \lambda_{p+1},\ldots,\lambda_{N}).
    \]
    Since $\mu+2\rho(\fraku \cap \frakp_\bbC)$ is $\triangle_c^+$-dominant, the representation $A_\frakq(\lambda)$ is isomorphic to $\pi_\lambda$ by the formula of the multiplicities of $K$-types.
    The statement (2) is a restatement of Lemma \ref{q_j=0}.
    We prove (3).
    Suppose $A_\frakq(\mu) \isom \pi_\lambda$.
    In this case, we have
    \[
    \nu_{p-p'+1} = [\mu_p - (p-q-1)/2, \mu_p - (p-q+1)/2 + p']
    \]
    and
    \[
    \nu_{p-p'+2} = [\mu_{p+1}-(p-q-1)/2-q',\mu_{p+1}-(p-q+1)/2].
    \]
    Since it is in the mediocre range, either $\mu_p - (p-q+1)/2 + p' \geq \mu_{p+1}-(p-q+1)/2$ or $\mu_p - (p-q-1)/2 \geq \mu_{p+1}-(p-q-1)/2-q'$ holds.
    This is equivalent to $\mu_p -\mu_{p+1} \geq -p'$ or $\mu_p -\mu_{p+1} \geq -q'$.
    By (2), we have $\mu_p = \lambda_p-q$ and $\mu_{p+1} = \lambda_{p+1}+p$.
    Then, the statement (3) follows from $p'(\lambda) \geq p'$ and $q'(\lambda) \geq q'$.
\end{proof}

    The signed tableaux for unitary lowest weight representations are as follows:

\begin{cor}\label{signed_tableau_lt_wt_rep}
    For a unitary lowest weight representation $\pi$, the signed tableau $\mathrm{AS}(\pi)$ has at most two columns.
    If the tableau $\mathrm{AS}(\pi)$ has only one column, $\pi$ is a character.
    If $\mathrm{AS}(\pi)$ has two columns, the signs are arranged in the order of $+$ and $-$ in rows with two boxes.
    In particular, the tableau $\mathrm{AS}(\pi)$ is uniquely determined by its shape.
\end{cor}
\begin{proof}
    This follows from Lemma \ref{A_q_lt_wt_rep} and the definition of $(p,q)$-signed tableau for $A_\frakq(\lambda)$.
\end{proof}

\subsection{Nonvanishing criterion for \texorpdfstring{$A_\frakq(\lambda)$}{}}

    When $\frakq$ is a maximal $\theta$-stable parabolic subalgebra, the nonvanishing of $A_{\frakq}(\lambda)$ can be described as follows by Trapa's algorithm.
    
    \begin{lem}[{\cite[Lemma 2.6]{C-Huang_2024}}]\label{A_q_nonvanishing}
        Let $\frakq$ be a $\theta$-stable parabolic subalgebra corresponding to $\underline{d} = \{(p_1,q_1), (p_2,q_2)\}$ and $\lambda$ be a $\triangle_c^+$-dominant integral weight such that $A_{\frakq}(\lambda)$ is in the mediocre range.
        Then, $A_{\frakq}(\lambda)$ is non-zero if and only if
        \[
        \min\{p_1,q_{2}\} + \min\{q_1,p_{2}\} \geq \#(\nu_1 \cap \nu_{2}).
        \]
        Here, $\nu_i$ is the segments associated with $\underline{d}$ and $\lambda$.
    \end{lem}

    Take a holomorphic $\theta$-stable parabolic subalgebra associated with $\underline{d} = \{(p_i,q_i)_{1 \leq i \leq r}\}$.
    Let $j = j(\underline{d})$ be the minimal integer $i$ such that $a_{\leq i} \geq p$.
    The cohomological inductions $A_\frakq(\lambda)$ for a holomorphic $\theta$-stable parabolic subalgebra are lowest weight by \cite[Lemma 1.7]{1987_Adams}.
    We will show that the converse holds in Lemma \ref{lemma_cond_d}.
    We can calculate the nonvanishing conditions and tableaux of such cohomological induction as follows:

    \begin{lem}\label{lemma_nonvanishing}
        Let $\frakq$ be a holomorphic $\theta$-stable parabolic subalgebra corresponding to $\underline{d} = \{(p_i,q_i)_{1 \leq i \leq r}\}$ and $A_\frakq(\lambda)$ be a cohomological induction in the mediocre range.
        Then, $A_\frakq(\lambda)$ is nonzero if and only if
        \begin{itemize}
            \item $\nu_{<j}$ and $\nu_{>j}$ are multiplicity free, and
            \item $\#(\nu_j \cap \nu_{<j}) \leq q_j$ and $\#(\nu_j \cap \nu_{>j}) \leq p_j$.
        \end{itemize}
        When $A_\frakq(\lambda)$ is nonzero and $q_j \neq 0$, set $\nu_j = \{\nu_{j,1},\ldots,\nu_{j,a_j}\} (resp.~\nu_{<j} \sqcup \nu_{>j} = \{\sigma_1,\ldots,\sigma_{t}\}$) with $\nu_{j,1} > \cdots > \nu_{j,a_j}$ (resp.~$\sigma_1 > \cdots > \sigma_{t}$) and $\nu_j \cap (\nu_{<j} \sqcup \nu_{>j}) = \{\sigma_{f+1}, \ldots, \sigma_{g}\}$.
        Let $m = \min\{f, q_j-\#(\nu_j \cap \nu_{<j})\}$ and $i_0$ be the maximal integer such that $1 \leq i_0 \leq g-f$ and $\sigma_{f+i_0} \geq \nu_{j, m+i_0}$.
        The first column of the tableau $\Ann(A_\frakq)$ consists of
        \begin{align*}
            &\sigma_1,\sigma_2,\ldots,\sigma_f, \\
            &\sigma_{f+1},\sigma_{f+2},\ldots,\sigma_{f+i_0}, \nu_{j, m+i_0+1}, \nu_{j,m+i_0+2},\ldots,\nu_{j,g}, \nu_{j,g+1},\nu_{j,a_j},\\
            &\sigma_{f+p-m+1},\ldots,\sigma_t
        \end{align*}
        and the second column consists of
        \[
        \nu_{j,1}, \nu_{j,2} \ldots, \nu_{j,m}, \nu_{j,m+1},\nu_{j, m+2}, \ldots, \nu_{j,m+i_0}, \sigma_{f+i_0+1}, \sigma_{f+i_0+2}, \ldots, \sigma_g, \sigma_{g+1},\ldots,\sigma_{\min\{t, f+p-m\}}
        \]
        from top to bottom.
        Here, we understand that there is no box next to the box filled with $\nu_{j, a_j}$ if $t\leq f+p-m$ in the first column.
        In particular, if $m \neq 0$, then the $(1,2)$-th entry of $\Ann(A_\frakq(\lambda))$ is $\nu_{j,1}$.
    \end{lem}
    \begin{proof}
        When $q_j = 0$, the statement follows from Lemma \ref{q_j=0}.
        When $q_j \neq 0$, we compute the tableau $\Ann(A_\frakq(\lambda))$.
        Let $h$ be the integer such that $\nu_{<j} = \{\sigma_1,\ldots, \sigma_h\}$.
        Before Trapa's algorithm, the first column of the tableau $\Ann(A_{\frakq}(\lambda))$ consists of
        \[
        \sigma_1,\ldots,\sigma_h,\nu_{j, m'+1},\ldots,\nu_{j,a_j},\ldots
        \]
        and the second column consists of
        \[
        \nu_{j,1},\nu_{j,2},\ldots,\nu_{j,m'}, \sigma_{h+1},\ldots
        \]
        from top to bottom.
        Here, $m' = \min\{h, q_j\}$.
        Let $\nu_{j-1} = \{\sigma_k,\ldots,\sigma_h\}$ and $S_i$ be the tableau consisting of entries with elements in $\nu_i$ defined in Trapa's algorithm.
        Consider $R = S_{j-1} \sqcup S_j$.
        By Trapa's algorithm, we have $R=R'$ if $\sigma_h \geq \nu_{j, m'}$ or $\nu_{j-1} \not\subset \nu_j$.
        In this case, $R$ is equivalent to zero when $q_j < \#(\nu_{j-1} \cap \nu_j)$.
        If not, the first column of $R'$ consists of
        \[
        \nu_{j, m'-h+k},\ldots,\nu_{j, m'}, \nu_{j, m'},\ldots
        \]
        and the second column consists of
        \[
        \nu_{j,1},\ldots,\nu_{j, m'-h-k-1},\sigma_{k},\ldots,\sigma_{h},\sigma_{h+1},\ldots
        \]
        from top to bottom.
        Note that if $\sigma_{h} < \nu_{j,m'}$, then $\sigma_k<\nu_{j,m'-h+k}$.
        Trapa's algorithm for $R = S_{j-2}' \sqcup S_{j-1}'$ is similar and $R$ is equivalent to zero if $q_j - \#(\nu_{j-1}) < \#(\nu_{j-2} \cap \nu_{j-1}')$.
        Also, the algorithm for $S_{j}' \sqcup S_{j+1}$ is similar.
        Continuing this procedure to the end, one obtains the lemma.
    \end{proof}

    The following statements, which are useful to compute $A_\frakq(\lambda)$, show that when the associated segments are linked, we may replace the linked segments or, conversely, partition them into smaller segments.
    
    \begin{cor}\label{cor_transform}
    Under the same notation as in Lemma \ref{lemma_nonvanishing}, suppose $A_\frakq(\lambda)$ is nonzero.
    Let $\nu'_1 \sqcup \cdots \sqcup \nu'_k = \nu_{<j}$ be a partition of $\nu_{<j}$ with $\nu'_1 > \cdots > \nu'_k$ and $\nu''_1 \sqcup \cdots \sqcup \nu''_\ell = \nu_{>j}$ be a partition of $\nu_{>j}$ with $\nu''_1 > \cdots > \nu''_\ell$.
    Let $\frakq'$ be the $\theta$-stable parabolic subalgebra associated with $\{(\#(\nu'_1), 0),\ldots, (\#(\nu'_k), 0), (p-\#(\nu_{<j}), q-\#(\nu_{>j})), (0, \#(\nu''_1)),\ldots, (0,\#(\nu_{\ell}''))\}$.
    Put $\pi = A(\frakq', \nu'_1\sqcup\cdots\sqcup\nu'_k\sqcup\nu_j\sqcup\nu''_1\sqcup\cdots\sqcup\nu''_{\ell})$.
    Then, if $\pi$ is in the mediocre range, we have $A_\frakq(\lambda) \isom \pi$.
    \end{cor}
    \begin{proof}
        The statement follows from Lemma \ref{lemma_nonvanishing}, since the tableau $A_\frakq(\lambda)$ does not depend on a partition of $\nu_{<j}$ and $\nu_{>j}$.
    \end{proof}

    \begin{cor}\label{cor_transform_2}
        Under the same notation as in Lemma \ref{lemma_nonvanishing}, suppose $A_\frakq(\lambda)$ is nonzero.
        When $\nu_{j-1} \subset \nu_j$ (resp.~$\nu_{j+1} \subset \nu_j$), let $\frakq'$ be the $\theta$-stable parabolic subalgebra corresponding to 
        \[\{(p_1,q_1), \ldots, (p_{j-2}, q_{j-2}), (p_{j} + p_{j-1}, q_{j}-p_{j-1}), (0, p_{j-1}),(p_{j+1},q_{j+1})\ldots, (p_r,q_r)\}
        \]
        \[
        \text{(resp.~$\{(p_1,q_1), \ldots, (p_{j-1}, q_{j-1}), (q_{j+1}, 0), (p_{j}-q_{j+1}, q_{j}+q_{j-1}),(p_{j+2}, q_{j+2})\ldots, (p_r,q_r)\}$)}
        \]
        and $\pi = A(\frakq', \nu_1,\ldots,\nu_j,\nu_{j-1}, \ldots, \nu_r)$ (resp.~$\pi = A(\frakq', \nu_1,\ldots,\nu_{j+1},\nu_{j}, \ldots, \nu_r)$).
        If $\pi$ is in the mediocre range, one has $A_{\frakq}(\lambda) \isom \pi$.
    \end{cor}
    \begin{proof}
        The statement follows from Lemma \ref{q_j=0} and Lemma \ref{lemma_nonvanishing} by the explicit computation of tableaux.
    \end{proof}

    \begin{rem}
        One can prove Corollary \ref{cor_transform} and Corollary \ref{cor_transform_2} by the induction in stages of $A_\frakq(\lambda)$ (cf.~\cite[\S 5.6]{C-Huang_2024} and \cite[Lemma 3.9]{Trapa_2001}).
        Note that Corollary \ref{cor_transform_2} is a special case of \cite[Proposition 4.9]{C-Huang_2024}.
    \end{rem}

\section{\texorpdfstring{$A$}{}-parameters and main theorem}
    We recall the M\oe glin-Renard's description of $A$-parameters in terms of cohomological inductions.
    We then state the main theorem of this paper.
    
\subsection{\texorpdfstring{$A$}{}-parameters}
    The $A$-parameters $\psi$ can be viewed as a formal sum
    \[
    \psi =  \bigoplus_{i=1}^r \chi_{t_i,s_i} \otimes S_{a_i},
    \]
    where $\chi_{t,s}$ is the character of $\bbC^\times$ defined by $z \mapsto (z/\overline{z})^{t/2}(z\overline{z})^{s/2}$, $S_m$ is the irreducible representation of $\SL_2(\bbC)$ with dimension $m$, and the triplets $(t,s,a)$ run over multisets on $(t,s,a) \in \bbZ \times \sqrt{-1}\,\bbR \times \bbZ_{>0}$.
    When $s=0$, we denote $\chi_t = \chi_{t,s}$.
    Note that this definition of $\chi_t$ differs slightly from that of \cite{Ichino_2022} but is the same as \cite{MR_consequence}.
    For an $A$-parameter $\psi = \bigoplus_i \chi_{t_i,s_i} \otimes S_{a_i}$, we say that $\psi$ is good if
    $(t_i+a_i+N)/2 \in \bbZ$ and $s_i = 0$ for any $i$. 
    Associated with an A-parameter $\psi$, we obtain the component group $\frakS_\psi$.
    It is isomorphic to a free $\bbZ/2\bbZ$-module
    \[
    \frakS_\psi = (\bbZ/2\bbZ)e_1 \oplus \cdots \oplus (\bbZ/2\bbZ)e_r.
    \]
    Let $\Pi(\psi)$ be the $A$-packet of $\psi$, i.e., the set of semisimple representations of $G$ satisfying the standard and twisted endoscopic character relations (cf.~\cite[\S 1.6]{AGIKMS_LIR_2024} and \cite[(1.6.1)]{KMSW_endoscopic_classification_unitary}).

    \begin{rem}
        In the usual definition, an $A$-parameter is a homomorphism $\psi_\bbR$ from the real Weil group $W_\bbR$ to the $L$-group of the unitary group.
        By the base change $\psi_\bbR|_{\bbC^\times}$, the parameter $\psi_\bbR$ can be identified with conjugate-selfdual parameters.
        For details, see \cite[Theorem 8.1]{2011_Gan_Gross_Prasad}.
    \end{rem}

\subsection{M{\oe}glin-Renard's construction of \texorpdfstring{$A$}{}-packets}

    Take a good $A$-parameter
    \[
    \psi = \bigoplus_{i=1}^r \chi_{t_i} \otimes S_{a_i}.
    \]
    We define the infinitesimal character $\chi_\psi$ of $\psi$ by the multiset
    \begin{align}\label{inf_char_A_param}
    \bigsqcup_{i=1}^r \left\{\frac{t_i+a_i-1}{2}, \frac{t_i+a_i-3}{2}, \ldots, \frac{t_i-a_i+1}{2}\right\} = \bigsqcup_{i=1}^r \left[\frac{t_i-a_i+1}{2}, \frac{t_i+a_i-1}{2}\right].
    \end{align}
Then, all the representations in $\Pi(\psi)$ have the same infinitesimal character with the Harish-Chandra parameter (\ref{inf_char_A_param}).
Following \cite[Th\'eor\`eme 1.1]{MR_consequence}, we describe the representations in $\Pi(\psi)$.
    Put
    \[
    \calD(\psi) = \left\{(p_i,q_i)_{i=1,\ldots,r} \in (\bbZ_{\geq 0})^2 \,\middle|\, \text{$p_i+q_i=a_i \neq 0$ for any $i$ and $\sum_{i=1}^r p_i = p, \sum_{i=1}^r q_i = q$}\right\}.
    \]
For $\underline{d} \in \calD(\psi)$, set
\[
x_{\underline{d}} = (\underbrace{r, \ldots, r}_{p_1}, \ldots, \underbrace{1, \ldots, 1}_{p_r},\underbrace{r, \ldots, r}_{q_1}, \ldots, \underbrace{1, \ldots, 1}_{q_r})
\]
and
\[
\lambda_{\underline{d}} = (\underbrace{\lambda_1, \ldots, \lambda_1}_{p_1}, \ldots, \underbrace{\lambda_r, \ldots, \lambda_r}_{p_r},\underbrace{\lambda_1, \ldots, \lambda_1}_{q_1}, \ldots, \underbrace{\lambda_r, \ldots, \lambda_r}_{q_r})
\]
where $\lambda_i = (t_i+a_i-N)/2+a_{<i}$ and $a_{<i} = \sum_{j<i}a_j$.
Note that there is a typo in \cite[(4.2)]{MR_consequence}: $(t_i+a_i-N)/2-a_{<i}$ should be $(t_i+a_i-N)/2+a_{<i}$.
We then consider the cohomological induction
\begin{align}\label{def_A_d_psi}
\mathscr{A}_{\underline{d}}(\psi) = A_{\frakq(x_{\underline{d}})}(\lambda_{\underline{d}})
\end{align}
and a character $\vep_{\underline{d}}$ on $\frakS_\psi$ defined by
\[
\vep_{\underline{d}}(e_i)= (-1)^{p_ia_{<i}+q_i(a_{<i}+1)+a_i(a_i-1)/2}
\]
for any $e_i \in \frakS_\psi$.
The following is proved in \cite[Th\'{e}or\`{e}me 1.1]{MR_consequence}.
    We choose the same Whittaker datum $\frakw_1$ as \cite{MR_consequence} (cf.~\cite[Appendix A]{Atobe_2020_nonvanishing_theta}).

\begin{thm}
    Let $\psi$ be a good A-parameter with $\psi = \bigoplus_{i=1}^r \chi_{t_i} \otimes S_{a_i}$.
    Suppose that $t_1\geq \cdots \geq t_r$ and $a_i \geq a_{i+1}$ if $t_i = t_{i+1}$.
    We then have
    \[
    \Pi(\psi) = \{\mathscr{A}_{\underline{d}}(\psi) \mid \underline{d} \in \calD(\psi)\}.
    \]
    The character of $\frakS_\psi$ associated with $\mathscr{A}_{\underline{d}}(\psi)$ is equal to $\vep_{\underline{d}}$.
    Moreover, the multiplicity one holds in $\Pi(\psi)$.
\end{thm}

\begin{rem}
    The correspondence $\scrA_{\underline{d}}(\psi)$ to $\vep_{\underline{d}}$ depends on the choice of Whittaker datum.
    We may explicitly calculate the dependence.
    See \cite[Theorem 1.6.1]{KMSW_endoscopic_classification_unitary} for details.
    \end{rem}
    \begin{rem}
    The statement does not imply that the cohomological induction $\mathscr{A}_{\underline{d}}(\psi)$ is non-zero.
    Trapa \cite{Trapa_2001} gives an algorithm to determine whether cohomological inductions of $G$ are zero or not.
    Recently, Huang \cite{C-Huang_2024} and Chengyu \cite{Chengyu_2024} independently considered the nonvanishing of cohomologically induced representations of real unitary groups with Chengyu treating the ``nice" case and Huang the general case.
\end{rem}
\begin{rem}
    For general $\psi$, the $A$-packet can be described as the parabolic induction from the representations of certain good $A$-parameter $\psi_0 \subset \psi$.
    Then, any representation in $\Pi(\psi)$ has the infinitesimal character $\chi_\psi$, where $\chi_\psi$ is the infinitesimal character defined as (\ref{inf_char_A_param}).
\end{rem}

\subsection{Main theorem}

    For a good $A$-parameter $\psi = \bigoplus_{i=1}^r \chi_{t_i}\otimes S_{a_i}$, let $j = j(\psi)$ be the minimal number $i$ so that $a_{\leq i} = \sum_{\ell=1}^i a_\ell \geq p$.
    By Lemma \ref{lemma_cond_d} below, it suffices to consider $\underline{d_0}$ defined by
    \[
    \underline{d_0} = \underline{d_0}(\psi) = \{(a_1,0),\ldots,(a_{j-1},0), (p_j,q_j), (0,a_{j+1}),\ldots, (0,a_{r})\}
    \]
    where $p_j = p - a_{<j}, q_j = q - a_{>j}$.
    Let $\nu_i$ be the segment $[(t_i-a_i+1)/2,(t_i+a_i-1)/2]$ and 
    \[
    \nu_{<i} = \bigsqcup_{k < i} \nu_k, \qquad \nu_{>i} = \bigsqcup_{i < k} \nu_k.
    \]
    Here we consider the union as multisets.
    The multisets $\nu_{\leq i}$ and $\nu_{\geq i}$ are defined similarly.

    For a $\triangle_c^+$-dominant integral weight $\lambda = (\lambda_1,\ldots,\lambda_p, \lambda_{p+1},\ldots,\lambda_N)$, put
    \[
    p' = p'(\lambda) = \#\{i \mid \lambda_p = \lambda_i, 1 \leq i \leq p\}, \qquad q' = q'(\lambda) = \#\{i \mid \lambda_i = \lambda_{p+1}, p+1 \leq i \leq N\}.
    \]
    Set
    \[
    P = P(\lambda) = \{\lambda_{p}-(N-1)/2,\lambda_{p-1}-(N-1)/2+1\ldots,\lambda_1+(p-q-1)/2\}
    \]
    and
    \[
    Q = Q(\lambda) = \{\lambda_{N}+(p-q+1)/2, \lambda_{N-1}+(p-q+3)/2,\ldots, \lambda_{p+1}+(N-1)/2\}.
    \]
    The multiset $P \sqcup Q$ can be identified with the infinitesimal character of the lowest weight representation $\pi_\lambda$.
    We define the segments $P'$ and $Q'$ by
    \[
    P' = [\lambda_p-(N-1)/2, \lambda_p-(N+1)/2+p'], \qquad Q' = [\lambda_{p+1}+(N+1)/2-q',\lambda_{p+1}+(N-1)/2].
    \]
    Set $I = P' \cap Q'$.
    
    \begin{lem}\label{lemma}
        Let $\lambda = (\lambda_1,\ldots,\lambda_{N})$ be a $\triangle_c^+$-dominant integral weight and $\pi_\lambda$ be the irreducible lowest weight representation of lowest $K$-type $\lambda$.
        If $\pi_\lambda \in \Pi(\psi)$, the parameter $\psi$ is good and $\chi_\psi = \chi_{\pi_\lambda}$.
        Moreover, if $\scrA_{\underline{d}}(\psi) \isom \pi_\lambda$, we have $\underline{d} = \underline{d_0}$.
    \end{lem}
    \begin{proof}
        The lemma will be proved in Lemma \ref{lemma_cond_d}.
    \end{proof}

    The following is the main theorem of this paper.
    Note that the nonvanishing condition $\scrA_{\underline{d_0}}(\psi)$ is determined in Lemma \ref{lemma_nonvanishing} (see also Corollary \ref{psi_to_lambda}).
    
    \begin{thm}\label{main}
        Let $\lambda = (\lambda_1,\ldots,\lambda_{N})$ be a $\triangle_c^+$-dominant integral weight and $\pi_\lambda$ be the irreducible lowest weight representation of lowest $K$-type $\lambda$.
        Let $\psi$ be a good $A$-parameter with $\chi_\psi = \chi_{\pi_\lambda}$ such that $\scrA_{\underline{d_0}}(\psi)$ is nonzero.
        We then obtain the following:
        \begin{enumerate}
            \item If $N- p' \leq \lambda_{p}-\lambda_{p+1} < N-q'$, the packet $\Pi(\psi)$ contains $\pi_\lambda$ if and only if $[\lambda_p-(N-1)/2, \lambda_{p+1}+(N-1)/2] \subset \nu_j \subset P'$.
            \item If $N-q' \leq \lambda_{p}-\lambda_{p+1} < N-p'$, the packet $\Pi(\psi)$ contains $\pi_\lambda$ if and only if either
            \begin{itemize}
                \item $\nu_{\leq j} = P$, or
                \item $[\lambda_{p}-(N-1)/2, \lambda_{p+1}+(N-1)/2] \subset \nu_j \subset Q'$.
            \end{itemize}
            \item If $N- p', N-q' \leq \lambda_{p}-\lambda_{p+1}$, the packet $\Pi(\psi)$ contains $\pi_\lambda$ if and only if either
            \begin{itemize}
                \item $P \subset \nu_{\leq j} \subset P \sqcup I$, or
                \item $I \subset \nu_{j} \subset Q'$.
            \end{itemize}
            \item If $\lambda_{p}-\lambda_{p+1} < N- p', N-q'$, the packet $\Pi(\psi)$ contains $\pi_\lambda$ if and only if $[\lambda_p-(N-1)/2, \lambda_{p+1}+(N-1)/2] = \nu_j$.
        \end{enumerate}
    \end{thm}

    Our proof, provided in section \ref{proof_main}, is based on the explicit computation of $K$-types and the associated tableaux of $\scrA_{\underline{d}}(\psi)$.
    More precisely, the if part follows Lemma \ref{A_q_lt_wt_rep}, Corollary \ref{cor_transform} and Corollary \ref{cor_transform_2}.
    For the only if part, we divide the cases into $q_j = 0$ or $q_j \neq 0$.
    When $q_j=0$, the statement follows from Lemma \ref{q_j=0}.
    When $q_j \neq 0$, we calculate the associated tableau $\Ann(\scrA_{\underline{d}}(\psi))$ in Lemma \ref{lemma_N-p'<a-b<N-q'}, \ref{lemma_N-q'<a-b<n-p'}, and \ref{lemma_a-b<N-p',N-q'_2}.
    Then, the theorem follows.

    As a consequence of Theorem \ref{main}, we can determine the lowest $K$-type of the lowest weight representation in $\Pi(\psi)$.
    \begin{thm}\label{psi_to_lambda}
        Let $\psi = \bigoplus_{i=1}^r \chi_{t_i} \otimes S_{a_i}$ be a good $A$-parameter.
        The representation $\Pi(\psi)$ contains a non-zero unitary lowest weight representation if and only if both
        \begin{itemize}
            \item $\nu_{<j}$ and $\nu_{>j}$ are multiplicity free, and
            \item $\#(\nu_j \cap \nu_{>j}) \leq p_j$ and $\#(\nu_j \cap \nu_{<j}) \leq q_j$.
        \end{itemize}
        When $\Pi(\psi)$ contains a nonzero unitary lowest weight representation $\pi$ in $\Pi(\psi)$, the lowest $K$-type $\lambda$ of $\pi$ is given as follows:
        \begin{enumerate}
            \item When $q_j = 0$, the lowest $K$-type $\lambda$ of $\pi$ satisfies $P(\lambda) = \nu_{\leq j}$ and $Q(\lambda) = \nu_{>j}$.
            \item When $p_j = \#(\nu_j \cap \nu_{>j})$ and $q_j \neq 0$, the lowest $K$-type $\lambda$ of $\pi$ satisfies $P(\lambda) = \nu_{<j} \sqcup (\nu_{j} \cap \nu_{>j})$ and $Q(\lambda) = \nu_{\geq j} \setminus (\nu_{j} \cap \nu_{>j})$.
            \item When $q_j = \#(\nu_j \cap \nu_{<j}) \neq 0$, the lowest $K$-type $\lambda$ of $\pi$ satisfies $P(\lambda) = \nu_{\leq j}\setminus (\nu_{<j} \cap \nu_j) = \{\sigma_1, \ldots,\sigma_p\}$ and $Q(\lambda) = (\nu_j \cap \nu_{<j}) \sqcup \nu_{>j}$.
            \item When $p_j \neq \#(\nu_j \cap \nu_{>j})$ and $q_j \neq \#(\nu_j \cap \nu_{<j})$, set $\nu_{<j} \sqcup \nu_{>j} = \{\sigma_1,\ldots,\sigma_{N-\#(\nu_j)}\}$.
            Let $i_0$ be the minimal integer such that $1 \leq i_0 \leq \#(\nu_j)$ and $\#(\nu_j) - i_0 + 1 + \#\{x \in\nu_{<j} \sqcup \nu_{>j} \mid x > \nu_{j,i_0}\}  = p$.
            Then, the lowest $K$-type $\lambda = (\lambda_1,\ldots,\lambda_N)$ of $\pi$ is given by
            \begin{align*}
            \lambda_i
            =
            \begin{cases}
                \sigma_i -(p-q+1)/2+i & \text{if $i < p-\#(\nu_j)+i_0$.}\\
                \nu_{j, 1} + (N+1)/2 -\#(\nu_j) & \text{if $p-\#(\nu_j)+i_0 \leq i \leq p$.}\\
                \nu_{j,1} - (N-1)/2 &\text{if p+1 $\leq i \leq p+i_0-1$.}\\
                \sigma_{i-\#(\nu_j)}-(N+1)/2-p+i & \text{if $p + i_0 \leq i$.}
            \end{cases}
            \end{align*}
        \end{enumerate}
    \end{thm}
    \begin{proof}
        To show the nonvanishing condition, it suffices to consider the case $\underline{d} = \underline{d_0}$ by Lemma \ref{lemma}.
        For $\scrA_{\underline{d_0}}(\psi)$, the nonvanishing condition is already proved in Lemma \ref{lemma_nonvanishing}.
        Suppose that $\scrA_{\underline{d_0}}(\psi)$ is nonzero.
        Then, if $q_j = 0$, the statement follows from Lemma \ref{q_j=0}.
        Under the conditions in (2) and (3), the statement follows from Lemma \ref{A_q_lt_wt_rep}, Corollary \ref{cor_transform} and Corollary \ref{cor_transform_2}.
        Under the conditions in (4), we have $\lambda_p - \lambda_{p+1} < N - p'(\lambda), N-q'(\lambda)$ by Theorem \ref{main}.
        In this case, such $i_0$ exists and the lowest $K$-type is given by Lemma \ref{A_q_lt_wt_rep}.
        We then have $\nu_j = [\lambda_{p}-(N-1)/2, \lambda_{p+1} + (N-1)/2]$ and $P(\lambda) = \{\nu_{j,i_0}, \nu_{j, i_0+1} \ldots,\nu_{j, \#(\nu_j)}\} \sqcup \{x \in \nu \setminus \nu_j \mid x>\nu_{j,i_0}\}$.
        The statement follows from Lemma \ref{A_q_lt_wt_rep} (1).
    \end{proof}

\section{Proof of main theorem}\label{proof_main}
    In this section, we prove the main theorem.
    The statement in Lemma \ref{lemma} will be proved in Lemma \ref{lemma_cond_d}.
    Theorem \ref{main} will be proved in \S \ref{case_N-p<a-b<N-q}, \ref{case_case_N-q<a-b<N-p}, \ref{case_N-p_N-q<a-b} and \ref{case_a-b<N-p_n-q}.

\subsection{Cohomological induction for holomorphic \texorpdfstring{$\theta$}{}-stable parabolic subalgebras}
    To show the main theorem, we prepare some lemmas for $A_\frakq(\lambda)$.
    The following is fundamental when we consider $\scrA_{\underline{d}}(\psi)$ that is isomorphic to a lowest weight representation.
    
    \begin{lem}\label{lemma_cond_d}
        Let $\psi = \bigoplus_{i=1}^r \chi_{t_i,s} \otimes S_{a_i}$ be an $A$-parameter.
        If $\Pi(\psi)$ contains an irreducible lowest weight representation $\pi$, the parameter $\psi$ is good and $\chi_\psi = \chi_\pi$.
        Moreover, if $\scrA_{\underline{d}}(\psi) \in \Pi(\psi)$ is nonzero and lowest weight, there exists $j$ such that $q_i = 0$ for any $i < j$ and $p_\ell = 0$ for any $\ell > j$, i.e., $\underline{d} = \underline{d_0}$ and $\frakq_{\underline{d}}$ is holomorphic.
    \end{lem}
    \begin{proof}
        The statement (1) follows from the fact that the representations in $\Pi(\psi)$ have the infinitesimal character $\chi_\psi$.
        For (2), consider the signed tableau $\Ann(\scrA_{\underline{d}}(\psi))$.
        If there exist different integers $k$ and $\ell$ with $p_kq_k p_\ell q_\ell \neq 0$, the signed tableau $\Ann(\scrA_{\underline{d}}(\psi))$ satisfies either
        \begin{itemize}
            \item there exists a row with three or more boxes, or
            \item there exists a row with two boxes arranged in the order of $-$ and $+$.
        \end{itemize}
        Then $\scrA_{\underline{d}}(\psi)$ is not of a unitary lowest weight representation by Corollary \ref{signed_tableau_lt_wt_rep}.
        Hence, the assertion holds.
    \end{proof}

    By Lemma \ref{lemma_cond_d}, it suffices to consider $\underline{d} = \underline{d_0}$.
    We already described the nonvanishing conditions of such cohomological inductions in Lemma \ref{lemma_nonvanishing}.
    In the following lemma, we will investigate the necessary condition that $\scrA_{\underline{d_0}}(\psi)$ is isomorphic to $\pi_\lambda$ with a given $\lambda$.
    A key point here is that the cohomological induction $\scrA_{\underline{d}}(\psi)$ is in the weakly fair range.

    \begin{lem}\label{basic_properties_of_d_0}
        Take an irreducible unitary lowest weight representation $\pi_\lambda$ with lowest $K$-type $\lambda$.
        Let $\psi$ be an $A$-parameter with $\pi_\lambda \in \Pi(\psi)$.
        Let $\nu_i$ be the segments associated with $\psi$ and $\underline{d_0}$.
        We then have the following:
        \begin{enumerate}
            \item As multisets, $\nu_{<j}$ and $\nu_{>j}$ are multiplicity free.
            \item If $\nu_j \subset \nu_{>j}$, then $q_j = 0$.
            \item $\nu_j \not\subset \nu_{<j}$.
            \item $I \subset \nu_{<j} \sqcup \nu_{>j}$.
            \item If $\nu_{<j} \cap \nu_{>j} \neq \emptyset$, then $\nu_j \subset \nu_{>j}$ and $q_j = 0$.
            \item If $\nu_{<j} \cap \nu_{>j} \cap I \neq \emptyset$ and $I \cap \nu_j \neq \emptyset$, then $\nu_j\subset I \subset \nu_{>j}$ and $q_j = 0$.
            \item If $\nu_{<j} \cap \nu_{>j} \cap I = \emptyset$, then $I \subset \nu_j$.
            \item If $I \neq \emptyset$, then $I \cap \nu_j \neq \emptyset$.
        \end{enumerate}
    \end{lem}
    \begin{proof}
        The statements (1), (2), and (3) follow immediately from Lemma \ref{lemma_nonvanishing}.
        Note that $p_j \neq 0$ by definition of $j$.
        For (4), consider the multiplicities of each element in $\nu$.
        The multiplicities of elements in $I$ are two in $\nu$.
        Thus, $\nu_{<j} \sqcup \nu_{>j}$ contains $I$, since $\nu_j$ is a set.

        Set $I = [x,y]$ and $\nu_{j} = [\alpha, \beta]$.
        By $\pi_\lambda \in \Pi(\psi)$, one has $P \sqcup Q = \nu$.
        For (5), take $t \in \nu_{<j} \cap \nu_{>j} \neq \emptyset$.
        Since any element in $\nu$ has multiplicity at most two, we have either
        \begin{itemize}
            \item $t < \alpha \leq \beta$; or
            \item $\alpha \leq \beta < t$.
        \end{itemize}
        In the former case, the minimum member in $\nu_{j-1}$ is less than $\alpha$.
        Then, the maximum member in $\nu_{j-1}$ is greater than $\beta$ by the weakly fair property.
        In particular, we have $\nu_j \subset \nu_{j-1} \subseteq \nu_{<j}$.
        However, this cannot happen by (3).
        In the latter case, by the same proof, we have $\nu_j \subset \nu_{j+1} \subseteq \nu_{>j}$ and then $q_j = 0$ by (2).
        
        For (6), assume $\nu_{<j} \cap \nu_{>j} \cap I \neq \emptyset$ and $\nu_j \cap I \neq \emptyset$.
        Since any element in $I$ has multiplicity two in $\nu$, one has $\nu_{<j} \cap \nu_{>j} \cap I = I \setminus (I \cap \nu_j)$ by (1).
        Moreover, $I \cap \nu_j$ is a segment since $I$ and $\nu_i$ are segments.
        Hence, $\nu_{<j} \cap \nu_{>j} \cap I$ contains $x$ or $y$.
        When $y \in \nu_{<j} \cap \nu_{>j}$, we denote by $z$ the minimal member in $\nu_{<j} \cap \nu_{>j} \cap I$.
        The maximal member in $\nu_j \cap I$ is $z-1$.
        By the weakly fair property, the set $\nu_{<j}$ does not contain $z-1$.
        We thus have $z-1 \in \nu_{>j}$.
        Then, the set $\nu_{>j}$ contains $I$ and in particular, $\nu_j \subset I \subset \nu_{>j}$.
        Here, we use the fact that the real numbers $x-1$ and $y+1$ have multiplicity at most one in $\nu$.
        When $x \in \nu_{<j} \cap \nu_{>j}$, one has $\nu_j \subset I \subset \nu_{<j}$ by the same discussion.
        This case does not happen by (3).
        Hence, we have $\nu_j \subset I \subset \nu_{>j}$ and then $q_j = 0$.

        The statement (7) follows immediately from the fact that the multiplicities of elements in $I$ are two.

        For (8), suppose $I \neq \emptyset$ and $I \cap \nu_j = \emptyset$.
        By the proof of (5), we have $\nu_j \subset \nu_{>j}$ and $\alpha \leq \beta < x$.
        Then, $q_j = 0$, and there exists an element $t$ in $P' \sqcup Q'$ with multiplicity two such that $t < x$.
        The existence of $t$ implies $\lambda_p - \lambda_{p+1} < N-q'$.
        By Lemma \ref{A_q_lt_wt_rep} (3), we may assume $N-p' \leq \lambda_p - \lambda_{p+1}$.
        Then, the set $I$ is equal to $Q'$.
        In this case, $x-1 \not\in Q$, but $x-1 \in P \cap \nu_{>j}$ by definition.
        Moreover, the segment $\nu_{>j}$ contains the set $Q \setminus Q'$ and $I = Q'$, and then $Q \subset \nu_{>j}$.
        Hence, $\#\nu_{>j} \geq \#\{x-1\} + \#Q \geq q+1$.
        This contradicts the definition of $j$.
        Hence $I \cap \nu_j \neq \emptyset$.
        This completes the proof.
    \end{proof}

    In the following sections, we complete the proof of Theorem \ref{main}.
    
\subsection{Proof of main theorem: the case \texorpdfstring{$N-p'\leq \lambda_p-\lambda_{p+1} < N-q'$}{}}\label{case_N-p<a-b<N-q}
    We first show the only if part.
    By the assumption, one has $I = Q'$.
    Put $\nu_j = [\alpha, \beta]$.

    When $q_j = 0$, by Lemma \ref{q_j=0}, we have $\nu_{\leq j} = P$ and $I = Q' \subset\nu_{> j} = Q$ if $\scrA_{\underline{d_0}}(\psi) \isom \pi_\lambda$.
    In this case, one has $\alpha = \lambda_p - (N-1)/2$.
    Note that by the weakly fair property, we have $\beta \geq \lambda_{p+1} + (N-1)/2$.
    Hence, the segment $\nu_j$ contains $[\lambda_{p}-(N-1)/2, \lambda_{p+1}+(N-1)/2]$.

    When $q_j \neq 0$, the segment $\nu_j$ is not contained in $\nu_{>j}$ by Lemma \ref{basic_properties_of_d_0} (1).
    Also, $\nu_j$ contains $I$ by Lemma \ref{basic_properties_of_d_0} (7).
    Set $\nu_1' = \nu_{<j} \setminus (\nu_j \cap \nu_{<j}), \nu_{2}' = \nu_j \cap \nu_{<j}, \nu_3' = \nu_j, \nu_4' = \nu_j \cap \nu_{>j}$ and $\nu_5' = \nu_{>j} \setminus (\nu_j \cap \nu_{>j})$.
    Let $\frakq'$ be the $\theta$-stable parabolic subalgebra associated with $\{(\#(\nu_1'), 0), (\#(\nu_2'), 0), (p_j, q_j), (0, \#(\nu_4')), (0, \#(\nu_5'))\}$.
    Then, the cohomological induction $A(\frakq', \nu_1',\ldots,\nu_5')$ is in the mediocre range and is isomorphic to $\scrA_{\underline{d_0}}(\psi)$ by Corollary \ref{cor_transform}.
    The only if part follows from the following lemma:

    \begin{lem}\label{lemma_N-p'<a-b<N-q'}
        With the above notation, if $A(\frakq', \nu_1',\ldots,\nu_5') \isom \pi_\lambda$, we then have $[\lambda_{p}-(N-1)/2, \lambda_{p+1}+(N-1)/2] \subset \nu_j \subset P'$.
    \end{lem}
    \begin{proof}
        By assumption $q_j \neq 0$, the segment $\nu_j$ contains $I$.
        Hence, it suffices to show $\alpha = \lambda_p - (N-1)/2$.
        Suppose $\alpha \geq \lambda_p-(N-1)/2$.
        By Lemma \ref{lemma_nonvanishing}, one has $\#(\nu_{<j} \cap \nu_j) \leq q_j = q - \#\nu_{>j}$.
        Since the set $(\nu_{<j} \cap \nu_j) \sqcup \nu_{>j}$ contains $Q$ in this case, one has $\#\nu_{>j} + \#(\nu_{<j} \cap \nu_{j}) \geq q$.
        Hence, we have $\nu_{>j} \sqcup (\nu_{<j} \cap \nu_j) = Q$.
        In other words, $\alpha \leq \lambda_{p}-(N-1)/2$.
        
        It remains to show $\alpha \geq \lambda_{p}-(N-1)/2$.
        Suppose $\lambda_{p} - (N-1)/2 > \alpha$.
        To show $A(\frakq', \nu_1',\ldots,\nu_5') \not\isom \pi_\lambda$ under this assumption, recall the tableau $\Ann(\pi_\lambda)$.
        By Lemma \ref{A_q_lt_wt_rep} (2), the first column of the tableau $\Ann(\pi_\lambda)$ consists of entries
        \[
        \lambda_1+(p-q-1)/2, \lambda_2+(p-q-3)/2, \ldots, \lambda_{p}-(N-1)/2, \ldots
        \]
        and the second column consists of entries
        \[
        \lambda_{p+1}+(N-1)/2, \lambda_{p+2}+(N-3)/2, \ldots, \lambda_{p+\min\{p,q\}} + (N+1)/2-\min\{p,q\}
        \]
        from top to bottom.
        Note that this is invariant under the Trapa's algorithm.
        To describe the tableau $A(\frakq', \nu_1',\ldots,\nu_5')$, set $\nu \setminus \nu_j = \{\sigma_1,\ldots,\sigma_{N - \#(\nu_j)}\}$ and $\nu_j \cap (\nu_{<j} \sqcup \nu_{>j}) = \{\sigma_{f+1},\ldots,\sigma_g\}$ with $\sigma_1 \geq \cdots \geq \sigma_{N-\#(\nu_j)}$.
        By $\#(\nu_{<j}) + p_j + q_j - \#(\nu_{<j} \cap \nu_j)>p$ and $\#(\nu_{<j}) + p_j = p$, we have $q_j - \#(\nu_{<j} \cap \nu_{j})>0$.
        When there is no $x \in \nu$ with $\beta < x$, the first column of $\Ann(A(\frakq', \nu_1',\ldots,\nu_5'))$ consists of
        \[
        \beta = \lambda_1+(p-q-1)/2, \beta-1,\ldots, \alpha,\ldots
        \]
        from top to bottom.
        In particular, the entry next to $\lambda_{p}-(N-1)/2$ is $\lambda_{p}-(N+1)/2$.
        However, the entry next to $\lambda_{p}-(N-1)/2$ in the first column of $\Ann(\pi_\lambda)$ does not equal $\lambda_{p}-(N+1)/2$.
        Indeed, if $q\leq p$, there is no such box.
        If $q>p$, the entry is $\lambda_{2p+1}+(N+1)/2-(p+1) = \lambda_{2p+1} -(p-q+1)/2$.
        We then have
        \begin{align*}
            \lambda_{p}-(N+1)/2 - (\lambda_{2p+1} -(p-q+1)/2) 
            &= \lambda_p - \lambda_{2p+1} -q\\
            &= \lambda_p - \lambda_{p+1} - (N-p) + (\lambda_{p+1} - \lambda_{2p+1})\\
            &\geq  \lambda_{p+1} - \lambda_{2p+1}.
        \end{align*}
        For the last inequality, we use $P=P'$ and $N-p' \leq \lambda_{p}-\lambda_{p+1}$.
        By $N-p = N - p' < N-q'$, one has $q' < p$ and then $\lambda_{p+1} - \lambda_{2p+1} > 0$.
        Hence, the tableaux $\Ann(\pi_\lambda)$ and $\Ann(A(\frakq', \nu_1',\ldots,\nu_5'))$ are different and in particular, the representations are different.
        We may assume that there exists $x \in \nu$ such that $x > \beta$.
        Put $f=\#\{x \in \nu \mid x>\beta\}$.
        Let $m= \min\{f, q_j - \#(\nu_j \cap \nu_{<j})\}$ and $i_0$ be the maximal integer such that $1\leq i_0 \leq g-f$ and $\sigma_{f+i_0} \geq \nu_{j, m+i_0}$.
        Here, $\nu_j = \{\nu_{j,1},\ldots,\nu_{j, a_j}\}$ with $\nu_{j,1} > \cdots > \nu_{j,a_j}$.
        By assumption, $m$ is positive.
        By Lemma \ref{lemma_nonvanishing}, the $(1,2)$-th entry in $\Ann(A(\frakq', \nu_1',\ldots,\nu_5'))$ is $\beta$.
        Hence, we have $\beta = \lambda_{p+1}+(N-1)/2$.
        This shows $i_0 \geq \#I$.
        The second column of $\Ann(A(\frakq', \nu_1',\ldots,\nu_5'))$ consists of
        \[
        \nu_{j,1}, \nu_{j,2}, \ldots, \nu_{j, m+i_0}, \ldots
        \]
        from top to bottom.
        In particular, the entry next to $\nu_{j, \#(I)} = \lambda_{p+1}+(N+1)/2-q'$ is $\nu_{j, \#(I)+1} = \nu_{j, \#(I)}-1$.
        Note that $\lambda_{p+1}+(N-1)/2-q'$ is in $\nu$ by $\lambda_p-\lambda_{p+1} < N-q'$.
        However, in the second column of $\Ann(\pi_\lambda)$, the entry next to $\lambda_{p+1}+(N+1)/2-q'$ is $\lambda_{p+q'+1}+(N-1)/2-q' < \nu_{j, \#(I)+1}$, if it exists.
        Hence, the representation $\lambda_{p+1}+(N+1)/2-q'$ is not isomorphic to $\pi_\lambda$ since the associated $\nu$-antitableau tableaux are different.
        This completes the proof.
    \end{proof}

    We show the converse.
    Suppose that $\scrA_{\underline{d_0}}(\psi)$ satisfies the condition $[\lambda_p-(N-1)/2, \lambda_{p+1}+(N-1)/2] \subset \nu_j \subset P'$.
    The nonvanishing of $\scrA_{\underline{d_0}}(\psi)$ follows from Lemma \ref{lemma_nonvanishing}.
    Since the multiset $\nu \setminus \nu_j = \nu_{<j} \sqcup \nu_{>j}$ is multiplicity free, the representation $\scrA_{\underline{d_0}}(\psi)$ is isomorphic to $\pi_\lambda$ by Lemma \ref{A_q_lt_wt_rep}, Corollary \ref{cor_transform} and Corollary \ref{cor_transform_2}.
     
\subsection{Proof of main theorem: the case \texorpdfstring{$N-q' \leq \lambda_p-\lambda_{p+1} < N-p'$}{}}\label{case_case_N-q<a-b<N-p}
    We first show the only if part.
    Let $P'' = (P \cap Q') \setminus I$ and $p'' = \#(P'')$.
    Note that one has $I = P'$ and $I \cap \nu_j \neq \emptyset$ by Lemma \ref{basic_properties_of_d_0} (8).
    When $q_j = 0$, by Lemma \ref{q_j=0}, the representation $\scrA_{\underline{d_0}}(\psi)$ is isomorphic to $\pi_\lambda$ if and only if $\nu_{\leq j} = P$.

    We consider the case where $q_j \neq 0$.
    Then, the multiset $\nu_{<j} \sqcup \nu_{>j}$ is multiplicity free and $\nu_j$ is not contained in $\nu_{>j}$.
    Set $\nu_1' = \nu_{<j} \setminus (\nu_j \cap \nu_{<j}), \nu_{2}' = \nu_j \cap \nu_{<j}, \nu_3' = \nu_j, \nu_4' = \nu_j \cap \nu_{>j}$ and $\nu_5' = \nu_{>j} \setminus (\nu_j \cap \nu_{>j})$.
    Let $\frakq'$ be the $\theta$-stable parabolic subalgebra associated with $\{(\#(\nu_1'), 0), (\#(\nu_2'), 0), (p_j, q_j), (0, \#(\nu_4')), (0, \#(\nu_5'))\}$.
    Then, the cohomological induction $A(\frakq', \nu_1',\ldots,\nu_5')$ is in the mediocre range and is isomorphic to $\scrA_{\underline{d_0}}(\psi)$ by Corollary \ref{cor_transform}.
    The statement follows from the following lemma:

    \begin{lem}\label{lemma_N-q'<a-b<n-p'}
        With the above notation, if $A(\frakq', \nu_1',\ldots,\nu_5') \isom \pi_\lambda$, we then have $[\lambda_{p}-(N-1)/2, \lambda_{p+1}+(N-1)/2] \subset \nu_j \subset Q'$. 
    \end{lem}
    \begin{proof}
        Put $\nu_j = [\alpha, \beta]$.
        Since $\nu_j$ contains $I$, one has $\alpha \leq \lambda_{p}-(N-1)/2$.
        It remains to show $\beta = \lambda_{p+1}+(N-1)/2$.
        Suppose that $\beta < \lambda_{p+1}+(N-1)/2$.
        In this case, we have $p_j \geq \#(\nu_j \cap \nu_{>j})$ and the multiset $\nu_{\leq j}$ contains a set $P$ properly, since $\nu_{<j}$ contains $\{x \in \nu \mid x > \beta\}$.
        We compute $\#(\nu_j \cap \nu_{>j}) - p_j$.
        By $\#(\nu_j \cap \nu_{>j}) = \#I + \#P'' - \#(\nu_{<j}\cap\nu_j)$, one has
        \begin{align*}
            \#(\nu_j \cap \nu_{>j}) - p_j 
            &= p'+p''-\#(\nu_j \cap \nu_{<j}) - p_j\\
            &= p'+p''-\#(\nu_j \cap \nu_{<j}) - (p-\#(\nu_{<j}))\\
            &= \#(\nu_{<j}) - \#(\nu_j \cap \nu_{<j}) -(p-p'-p'')>0.
        \end{align*}
        The last inequality follows from $\beta < \lambda_{p+1}+(N-1)/2$.
        Hence $A(\frakq', \nu_1',\ldots,\nu_5')$ is zero.
        This is a contradiction.
        Therefore we have $\beta \geq \lambda_{p+1}+(N-1)/2$.

        It remains to show $\beta \leq \lambda_{p+1}+(N-1)/2$.
        Put $f = \#\{x \in \nu \mid x > \lambda_{p+1} + (N-1)/2\}$.
        Suppose $\beta > \lambda_{p+1}+(N-1)/2$.
        By assumption, we have $q_j > \#(\nu_j \cap \nu_{<j})$ and $f \neq 0$.
        We recall the tableau $\Ann(\pi_\lambda)$.
        By Lemma \ref{A_q_lt_wt_rep}, the second column of the tableau $\Ann(\pi_\lambda)$ consists of
        \begin{align*}
        \underbrace{\lambda_{p+1} + (N-1)/2, \ldots, \lambda_{p+1}+(N+1)/2-(m+i_0)}_{m+i_0}, \sigma_{f+i_0+1}, \dots, \sigma_{\min\{t, f+p-m\}}.
        \end{align*}
        Here, $\nu \setminus Q' = \{\sigma_1,\ldots, \sigma_{t}\}, Q' \cap P = \{\sigma_{f+1},\ldots, \sigma_g\}, m= \min\{f,f+q'-p\}$ with $\sigma_1 > \cdots > \sigma_{N-q'}$, and, $i_0$ is the maximal integer such that $1 \leq i_0 \leq g-f$ and $\sigma_{f+i_0} \geq \lambda_{p+1} + (N+1)/2-(m+i_0)$.
        In particular, the $(1,2)$-th entry is $\lambda_{p+1}+(N-1)/2$.
        Note that in this case, one has $m>0$, and the number of boxes in the second column from the top to the box filled with $\lambda_{p}-(N-1)/2$ is greater than or equal to $m+p'+p''$.
        If there exists $x \in \nu$ with $x>\beta$, the $(1,2)$-th entry of $\Ann(A(\frakq', \nu_1,\ldots,\nu_5))$ is $\beta$ by $q_j > \#(\nu_j \cap \nu_{>j})$.
        Thus, the tableaux $\Ann(\pi_\lambda)$ and $\Ann(A(\frakq', \nu_1,\ldots,\nu_5))$ are different by the assumption $\beta > \lambda_{p+1}+(N-1)/2$.
        In other words, the representations $\pi_\lambda$ and $A(\frakq', \nu_1', \ldots, \nu'_5)$ are different.
        The remaining case is that there is no $x \in \nu$ with $x > \beta$.
        Then, the second column of $\Ann(\scrA_{\underline{d_0}}(\psi))$ consists of entries
        \[
        \lambda_{f+1}+\frac{p-q-1}{2}-f,\lambda_{f+2}+\frac{N-3}{2}-f, \ldots, \lambda_{p}-\frac{N-1}{2},\ldots
        \]
        from top to bottom.
        The number of boxes from the top to the box filled with $\lambda_p-(N-1)/2$ is $p''+p'$, which is different from that of $\Ann(\pi_\lambda)$.
        Therefore, the representation $\scrA_{\underline{d_0}}(\psi)$ is not isomorphic to $\pi_\lambda$ by $m>0$.
        This completes the proof.
        \end{proof}

        It remains to show the converse.
        Suppose that $\scrA_{\underline{d_0}}(\psi)$ satisfies $\nu_{\leq j} = P$.
        Then, $\scrA_{\underline{d_0}}(\psi) \isom \pi_\lambda$ by Lemma \ref{q_j=0}.
        Suppose next that $\scrA_{\underline{d_0}}(\psi)$ satisfies $[\lambda_{p}-(N-1)/2, \lambda_{p+1}+(N-1)/2] \subset \nu_j \subset Q'$.
        In this case, by explicit computation of $p_j$ and $q_j$, the representation $\scrA_{\underline{d_0}}$ is nonzero and isomorphic to $\pi_\lambda$ by Lemma \ref{A_q_lt_wt_rep}, Corollary \ref{cor_transform} and Corollary \ref{cor_transform_2}.

\subsection{Proof of main theorem: the case \texorpdfstring{$N-p', N-q' \leq \lambda_p-\lambda_{p+1}$}{}}\label{case_N-p_N-q<a-b}

    We first show the only if part.
    When $q_j =0$, by Lemma \ref{q_j=0}, we have $\nu_{\leq j} = P$ and $\nu_{>j} = Q$.
    In particular, $P \subset \nu_{\leq j} \subset P \sqcup I$.
    When $q_j \neq 0$, the segment $\nu_j$ contains $I$.
    When $\nu_j \cap (Q' \setminus I) = \emptyset$, we have $\nu_{\leq j} \subset P \sqcup I$ and $P \subset \nu_{\leq j}$ since $\nu_{>j}$ is multiplicity free by Lemma \ref{lemma_nonvanishing}.
    We may assume that $\nu_j \cap (Q' \setminus I) \neq \emptyset$ and $\lambda_{1}+(p-q-1)/2 > \lambda_{p+1}+(N-1)/2$.
    Consider the tableaux for  $\Ann(\scrA_{\underline{d_0}}(\psi))$ and $\Ann(\pi_\lambda)$.
    In this case, we have $\nu_j \subset P' \sqcup Q'$ and $q_j \neq 0$.
    We show that $\nu_j$ is contained in $Q'$ if $\scrA_{\underline{d_0}}(\psi) \isom \pi_\lambda$.
    Note that we have $q_j - \#(\nu_j \cap \nu_{<j}) > 0$.
    Indeed, by assumption, the multiset $\nu_{\leq j}$ contains $P$ properly and then $P \sqcup (\nu_j \cap \nu_{<j}) \subsetneq \nu_{\leq j}$.
    We then have $p + \#(\nu_j \cap \nu_{<j}) < \#\nu_{< j} + p_j + q_j$ and, in particular, $0 < q_j - \#(\nu_j \cap \nu_{<j})$ by $p_j + \#(\nu_{<j}) = p$.
    By Lemma \ref{A_q_lt_wt_rep}, the first column of $\Ann(\pi_\lambda)$ is 
    \[
    \lambda_1+(p-q-1)/2,\lambda_1+(p-q-3)/2, \ldots, \lambda_p-(N-1)/2,\ldots
    \]
    and the second column is
    \[
    \lambda_{p+1}+(N-1)/2,\lambda_{p+2}+(N-3)/2, \ldots, \lambda_{p+\min\{p,q\}}+(N+1)/2-\min\{p,q\}.
    \]
    Put $\nu_j = [\alpha, \beta]$.
    Suppose $\beta = \lambda_1+(p-q-1)/2$, then $P=P'$ and $\beta > \lambda_{p+1}+(N-1)/2$.
    The first column of $\Ann(\scrA_{\underline{d_0}}(\psi))$ consists of
    \[
    \beta, \beta-1,\ldots,\alpha,\ldots
    \]
    and the second column consists of
    \[
    \lambda_{p+1}+(N-1)/2, \lambda_{p+2}+(N-3)/2, \ldots, \lambda_{p+\ell}+(N+1)/2-\ell, \alpha-1,\ldots
    \]
    from top to bottom.
    Here, $\ell$ is the unique positive integer such that $\lambda_{p+\ell} +(N+1)/2-\ell = \lambda_{p}-(N-1)/2$.
    Note that in the second column of $\Ann(\scrA_{\underline{d_0}}(\psi))$, the box next to the box filled with $\lambda_{p+\ell}+(N+1)/2-\ell$ exists if and only if $\beta > \lambda_{p+1}+(N-1)/2$ and $\nu_j \sqcup I \neq \nu$.
    Contrary to this, in the second column of $\Ann(\pi_\lambda)$, the box next to the box filled with $\lambda_{p+\ell}+(N+1)/2-\ell$ exists if and only if $\lambda_1+(p-q-1)/2 > \lambda_{p+1}+(N-1)/2$ and $I \neq Q$.
    Hence, under our assumption, there exists a box next to the box filled with $\lambda_{p+\ell}+(N+1)/2-\ell$ in the second column of $\Ann(\pi_\lambda)$ and it is equal to $\lambda_{p+\ell} + (N-1)/2 - \ell$.
    We may additionally assume $\nu \neq \nu_j \sqcup I$.
    The entry next to $\lambda_{p+\ell} + (N+1)/2-\ell$ in the second column of $\Ann(\scrA_{\underline{d_0}}(\psi))$ is strictly less than $\lambda_{p+\ell} + (N+1)/2-\ell-1$.
    This shows that the tableaux $\Ann(\scrA_{\underline{d_0}}(\psi))$ and $\Ann(\pi_\lambda)$ are different.
    Suppose $\beta \neq \lambda_1+(p-q-1)/2$.
    Then, the $(1,2)$-th entry in $\Ann(\scrA_{\underline{d_0}}(\psi))$ is $\beta$ by $q_j > \#(\nu_j \cap\nu_{<j})$.
    Hence, we have $\beta = \lambda_{p+1}+(N-1)/2$ if $\scrA_{\underline{d_0}} \isom \pi_\lambda$.
    In other words, one has $I \subset \nu_j \subset Q'$ if $\scrA_{\underline{d_0}} \isom \pi_\lambda$.

    The converse follows from Lemma \ref{A_q_lt_wt_rep}, Lemma \ref{lemma_nonvanishing}, Corollary \ref{cor_transform} and Corollary \ref{cor_transform_2}.
    This completes the proof.

\subsection{Proof of main theorem: the case \texorpdfstring{$\lambda_p-\lambda_{p+1} < N-p', N-q'$}{}}\label{case_a-b<N-p_n-q}

    We first show the only if part. Put $\nu_j = [\alpha, \beta]$.
    In this case, one has $q_j > 0$ by Lemma \ref{A_q_lt_wt_rep} (3).
    Then, $\nu_j$ contains $I$ and $\nu_{<j} \sqcup \nu_{>j}$ is multiplicity free.
    Set $\nu_1' = \nu_{<j} \setminus (\nu_j \cap \nu_{<j}), \nu_{2}' = \nu_j \cap \nu_{<j}, \nu_3' = \nu_j, \nu_4' = \nu_j \cap \nu_{>j}$ and $\nu_5' = \nu_{>j} \setminus (\nu_j \cap \nu_{>j})$.
    Let $\frakq'$ be the $\theta$-stable parabolic subalgebra associated with $\{(\#(\nu_1'), 0), (\#(\nu_2'), 0), (p_j, q_j), (0, \#(\nu_4')), (0, \#(\nu_5'))\}$.
    Then, the cohomological induction $A(\frakq', \nu_1',\ldots,\nu_5')$ is in the mediocre range and is isomorphic to $\scrA_{\underline{d_0}}(\psi)$ by Corollary \ref{cor_transform}.
    The statement follows from the following lemma:

    \begin{lem}\label{lemma_a-b<N-p',N-q'_2}
        With the above notation, if $A(\frakq', \nu'_1,\ldots,\nu'_5) \isom \pi_\lambda$, one has $\nu_j = \nu_{3}' = [\lambda_{p}-(N-1)/2, \lambda_{p+1}+(N-1)/2]$.
    \end{lem}
    \begin{proof}
        The statement follows from the explicit calculation of the associated tableaux.
        Recall the tableau $\Ann(\pi_\lambda)$.
        Put $f = \#\{x \in \nu \mid x>\lambda_{p+1}+(N-1)/2\}$ and $m = \min\{f, f+N-(\lambda_p - \lambda_{p+1})-p\}$.
        When $f \neq 0$, the $(1,2)$-th entry in $\Ann(\pi_\lambda)$ is $\lambda_{p+1}+(N-1)/2$.
        Let $i_0$ be the maximal positive integer such that $\lambda_{f+i_0}+(p-q+1)/2 -(f+i_0) \geq \lambda_{p+1}+(N+1)/2-m-i_0$, if it exists.
        If there is no such $i_0$, set $i_0=0$.
        Then, the first column of $\Ann(\pi_\lambda)$ consists of 
        \begin{align*}
        &\underbrace{\lambda_{1}+\frac{p-q-1}{2}, \lambda_2+\frac{p-q-3}{2},\ldots, \lambda_{f}+\frac{p-q+1}{2}-f,}_f\\
        &\underbrace{\lambda_{f+1}+\frac{p-q+1}{2}-(f+1), \lambda_{f+2}+\frac{p-q+1}{2}-(f+2) \ldots, \lambda_{f+i_0}+\frac{p-q+1}{2}-(f+i_0),}_{i_0}\\
        &\lambda_{p+m+i_0+1}+\frac{N-1}{2}- (m+i_0), \lambda_{p+m+i_0+2}+\frac{N-1}{2}- (m+i_0+1), \ldots, \lambda_{p} - (N-1)/2, \ldots
        \end{align*}
        from top to bottom.
        The entry next to $\lambda_{p} - (N-1)/2$ is strictly less than $\lambda_{p}-(N-1)/2-1$ since $p-f$ is strictly greater than the number of elements in $\nu$ with multiplicity two.
        When $f = 0$, the $(1,2)$-th entry in $\Ann(\pi_\lambda)$ is the maximal member in $\nu$ with multiplicity two.
        This is greater than or equal to $\lambda_{p}-(N-1)/2+p'$.
        Then the first column consists of 
        \[
        \lambda_{p+1}+\frac{N-1}{2}, \lambda_{p+1}+\frac{N-3}{2},\ldots,\lambda_{p}-\frac{N-3}{2},\lambda_{p}-(N-1)/2,\ldots
        \]
        from top to bottom.
        
        We claim $\beta = \lambda_{p+1}+(N-1)/2$.
        Note that $q_j > \#(\nu'_{2} \sqcup \nu'_{4})$ if $A(\underline{d}', \nu'_1,\ldots,\nu'_5) \isom \pi_\lambda$ by Lemma \ref{A_q_lt_wt_rep} (3).
        Suppose first that $\beta < \lambda_{p+1}+(N-1)/2$.
        Then, there exists $x \in \nu'$ such that $x > \beta$.
        In this case, the $(1,2)$-th entry in $\Ann(A(\underline{d'}, \nu'_1,\ldots,\nu_5'))$ is $\beta$ by $q_j > \#(\nu'_{2} \sqcup \nu_{4}')$.
        If $A(\underline{d'}, \nu'_1,\ldots,\nu_5') \isom \pi_\lambda$, there exists no $x \in \nu$ with $x > \lambda_{p+1}+(N-1)/2$, i.e., $f=0$, and $\beta$ is the maximal member in $\nu$ with multiplicity two.
        Consider the number of boxes from the top to the box filled with $\lambda_{p+1}+(N-1)/2-q'$ in the first column.
        For $\pi_\lambda$, by $f=0$, this number is $q'$, but for $\scrA_{\underline{d_0}}(\psi)$, it is strictly less than $q'$ by $\beta < \lambda_{p+1}+(N-1)/2$.
        Hence, the representations $\pi_\lambda$ and $\scrA_{\underline{d_0}}(\psi)$ are different.
        Suppose next that $\beta > \lambda_{p+1}+(N-1)/2$.
        Then, $f>0$.
        If there exists $x \in \nu$ with $x > \beta$, the $(1,2)$-th entry in $\scrA_{\underline{d_0}}(\psi)$ is $\beta$.
        Then, one has $\beta = \lambda_{p+1}+(N-1)/2$.
        This is a contradiction.
        When there exists no $x \in \nu$ with $x > \beta$, we compare the $\nu$-antitableaux of $\pi_\lambda$ and $\scrA_{\underline{d_0}}(\psi)$.
        In this case, for $\pi_\lambda$, the number of boxes from the top to the box filled with $\lambda_{p+1}+(N+1)/2-q'$ in the first column is at most $\max\{-q+q'+(\lambda_p-\lambda_{p+1}), q'\}$, that is strictly less than $f+q'$ by $f \neq 0$.
        Note that $-q+q'+(\lambda_p-\lambda_{p+1})$ is equal to the integer $a$ such that $\lambda_a+(p-q+1)/2-a = \lambda_{p+1}+(N+1)/2-q'$.
        For $\scrA_{\underline{d_0}}(\psi)$, the number of boxes from the top to the box filled with $\lambda_{p+1}+(N+1)/2-q'$ in the first column is $f+q'$, since $\nu_j$ contains $I$.
        Hence, the tableaux are different.
        This concludes that $\beta$ is equal to $\lambda_{p+1}+(N-1)/2$ if $\pi_\lambda \isom \scrA_{\underline{d_0}}(\psi)$.
        In the following, we may assume $\beta = \lambda_{p+1}+(N-1)/2$.

        It remains to show $\alpha = \lambda_p - (N-1)/2$.
        If there exists no $x \in \nu$ with $x > \lambda_{p+1}+(N-1)/2$, the first column of $\Ann(\scrA_{\underline{d_0}}(\psi))$ consists of
        \[
        \lambda_{p+1}+(N-1)/2 = \beta,\ldots,\alpha,\ldots
        \]
        from top to bottom.
        Then, the entry next to $\alpha$ is strictly less than $\alpha-1$ by $q_j>\#(\nu_2' \sqcup \nu_4')$.
        For $\pi_\lambda$, in the first column, the entry next to $\lambda_{p}-(N-1)/2$ is strictly less than $\lambda_{p}-(N-1)/2-1$ and the entry next to $x$ with $\lambda_{p} - (N-1)/2 < x \leq \lambda_{p+1} + (N-1)/2$ is $x-1$.
        Hence, if $A(\underline{d}, \nu'_1,\ldots,\nu'_5) \isom \pi_\lambda$, we have $\alpha = \lambda_{p}-(N-1)/2$.
        We may assume that there exists $x \in \nu$ with $x > \lambda_{p+1}+(N-1)/2$, i.e., $f \neq 0$.
        Recall that by $\pi_\lambda \isom \scrA_{\underline{d_0}}(\psi)$, we have $p_j>\#(\nu_{4}')$.
        This shows that the entry next to $\alpha$ in the first column is strictly less than $\alpha-1$ if the box exists.
        By the description of $\Ann(\pi_\lambda)$ and the same discussion above, we have $\alpha = \lambda_{p}-(N-1)/2$.
        This completes the proof.
        \end{proof}

    For the converse, apply Lemma \ref{A_q_lt_wt_rep}(1), Corollary \ref{cor_transform} and Corollary \ref{cor_transform_2}.
    We then have $\pi_\lambda \isom \scrA_{\underline{d_0}}(\psi)$.
    This completes the proof.

\bibliographystyle{alpha}
\bibliography{ref}

\end{document}